\documentclass[11pt]{article}
\usepackage[margin=2.9cm]{geometry}

\usepackage{amssymb}
\usepackage{amsmath}
\usepackage{amsthm}
\usepackage{verbatim}
\usepackage{bm}
\usepackage{setspace}
\usepackage[inline,shortlabels]{enumitem}

\usepackage[font=small]{caption}

\RequirePackage[breaklinks=true,colorlinks,citecolor=blue,urlcolor=blue]{hyperref}

\usepackage{booktabs}       %
\usepackage{amsfonts}       %
\usepackage{nicefrac}       %
\usepackage{microtype}      %
\usepackage{diagbox}
\usepackage{mathrsfs}
\usepackage{algorithm2e}
\usepackage{mathtools}
\usepackage[title]{appendix}
\usepackage{graphicx}
\graphicspath{{./}}
\usepackage{subcaption,cancel}
\usepackage{mwe}
\usepackage[opnorm=op,matrixnorms]{arash_macros}

\usepackage[normalem]{ulem}

\usepackage{authblk}

\newcommand\Gf{G_f(\Theta^*)}
\newcommand\Gfs{G_f}
\newcommand\Sc{\mathtt{S}}

\newcommand\Xd{\msf X}
\newcommand\Nc{\mathcal N}
\newcommand\rhob{\overline\rho}
\newcommand\Cc{\mbb C}
\newcommand\Ec{\mc E}
\newcommand\taut{\widetilde{\tau}}

\newcommand{\Real}{\mathbb{R}}

\newcommand{\ie}{{\rm i.e.}}
\newcommand{\T}{^\top}
\newcommand{\Exp}{\mathbb{E}}
\newcommand{\Prob}{\mathbb{P}}

\newcommand{\half}{\frac{1}{2}}

\newcommand{\conv}{{\rm conv}}

\newcommand{\bin}{\{0,1\}}

\newcommand{\inner}[1]{\left< #1 \right>}
\newcommand{\frob}[1]{\|#1\|_F}

\newcommand{\onorm}[1]{\|#1\|_1}
\newcommand{\tvnorm}[1]{\|#1\|_{\rm TV}}
\newcommand{\tnorm}[1]{\|#1\|_2}

\newcommand{\bigO}{\mathcal{O}}

\newtheorem{theorem}{Theorem}[section]
\newtheorem{lemma}[theorem]{Lemma}
\newtheorem{proposition}[theorem]{Proposition}

\theoremstyle{definition}
\newtheorem{definition}{Definition}

\newtheorem*{definition*}{Definition}
\newtheorem*{remark*}{Remark}
\newtheorem*{example*}{Example}

 \newcommand\rem{R}

\newcommand{\bfone}{\mathbf{1}}

\newcommand{\bfu}{\mathbf{u}}

\def\bfone{{\bf 1}}
\renewcommand{\qed}{\nobreak \ifvmode \relax \else
      \ifdim\lastskip<1.5em \hskip-\lastskip
      \hskip1.5em plus0em minus0.5em \fi \nobreak
      \vrule height0.75em width0.5em depth0.25em\fi}

\newcommand{\like}{\mathcal{L}}
\newcommand{\bfR}{\mathbf{R}}
\newcommand{\Cov}{{\rm Cov}}
\newcommand{\herm}{^{\rm H}}

\newcommand{\mcE}{\mathcal{E}}
\newcommand{\msf}[1]{\mathsf{#1}}

\newcommand{\mc}[1]{\mathcal{#1}}

\newcommand{\wt}[1]{\widetilde{#1}}
\newcommand{\wh}[1]{\widehat{#1}}
\newcommand{\mbf}[1]{\mathbf{#1}}
\newcommand{\mbb}[1]{\mathbb{#1}}%

\newcommand{\set}[1]{\{#1\}}
\def\beq{\begin{equation}}
\def\eeq{\end{equation}}
\def\beqstar{\begin{equation*}}
\def\eeqstar{\end{equation*}}
\newcommand{\defeq}{\vcentcolon=}

\newcommand{\ThML}{\widehat{\Theta}_{\rm ML}}

\newcommand{\tonorm}[1]{\|#1\|_{1,1,1}}
\newcommand{\Th}{\Theta}

\newcommand{\mdot}{*}

\newcommand{\Ths}{\Theta^*}
\newcommand{\qt}{\tilde{q}}

\def\bm{\begin{matrix}}
\def\bbm{\begin{bmatrix}}
\def\ebm{\end{bmatrix}}
\newcommand\dkl{D_{\rm KL}} 

\title{High-Dimensional Bernoulli Autoregressive Process with Long-Range Dependence\footnote{A version of this work titled ``{\it Sparse Multivariate Bernoulli Processes in High Dimensions}'' appeared in proceedings  of AISTATS 2019, PMLR: Volume 89.}}

\author{Parthe Pandit$^\dagger$,\,\, Mojtaba Sahraee-Ardakan$^\dagger$,\,\, Arash A. Amini$^\#$,\\
Sundeep Rangan$^\diamondsuit$,\,\, Alyson K. Fletcher$^{\#,\dagger}$. 
}
\affil{$^\dagger$Department of Electrical and Computer Engineering, UCLA,\\$^\#$Department of Statistics, UCLA, \\
        $^\diamondsuit$Department of Electrical and Computer Engineering, NYU.}
\date{}
\begin{document}
\maketitle
\begin{abstract}
We consider the problem of estimating the parameters of a multivariate Bernoulli
process with auto-regressive
feedback in the high-dimensional setting where the number of samples available is much less than the number of parameters.
This problem arises in learning interconnections of networks of dynamical systems with spiking or binary-valued data.  We  allow the process to depend on its  past up to a lag $p$, for a general $p \ge 1$, allowing for more realistic modeling in many applications. We propose and analyze an $\ell_1$-regularized maximum likelihood estimator (MLE) under the assumption that the parameter tensor is approximately sparse.
Rigorous analysis of such estimators is made challenging by the dependent and
non-Gaussian nature of the process as well as the presence of the nonlinearities
and multi-level feedback. We derive precise upper bounds on the mean-squared estimation
error in terms of the number of samples, dimensions of the process, the lag $p$ and other key
statistical properties of the model. The ideas presented can be used in the
 high-dimensional analysis of regularized $M$-estimators for %
other sparse nonlinear and non-Gaussian processes with long-range dependence.

\end{abstract}

\newcommand{\Eta}{\mbf H}

\section{Introduction}

In many signal processing applications, the underlying time series may be modeled as a multivariate
Bernoulli process (MBP). For example, the spike trains from an ensemble of neurons can be modeled
as a collection of Bernoulli variables where for each neuron, at each time instant, the probability
of spiking could depend on the history of the spikes from the ensemble. Similarly, several other
signals such as activity in social networks, trends of stock prices in financial markets \cite{startz2008binomial,taylor2016using}, crime
occurrences in a metropolitan area \cite{mark2018network}, medical emergency call forecasting \cite{matteson2011forecasting}, climate dynamics \cite{guanche2014autoregressive} and certain physiological \cite{kowsar2017autoregressive} and biological processes \cite{katselis2018mixing} can all be encoded as multivariate autoregressive Bernoulli processes where the
history of the process affects the present outcome.

We are interested in developing generalized linear
models (GLM) to capture the behavior of  temporally-dependent MBPs. Such models  allow one
to not only make predictions about the future of the process, but also infer the relations among the
coordinates  (i.e.,  individual time series), based on proper estimates of the parameters
of the model. For example, in the neural spike train application, one can reveal a latent network among the
neurons (i.e., who influences whose firing) just from the observations of patterns of neural activity, a
task which is of significant interest in neuroscience~\cite{okatan2005analyzing,smith2003estimating,brown2004multiple}. Similarly, in the context of social networks, one might be
interested in who is influencing whom~\cite{raginsky2012sequential}.

In a GLM, one models an invertible link function of the conditional mean of the observed variables as a linear function of the covariates. In  time series analysis, the role of the covariates is played by the history of the evolving time series. For Gaussian random variables with the identity link function, this  leads to the classical Gaussian autoregressive (AR) process. Indeed, much of the focus in time series analysis has been on VAR (vector AR) processes \cite{basu2015regularized,cai2016estimating,mcmurry2015high,
	mei2017signal,ahelegbey2016sparse}, which provide significant richness as a rudimentary model. However, these models fail to capture higher order correlations among the variables and become harder to interpret for variables from discrete spaces.

Another fundamental class of processes is the MBP, for which very few results are known as compared to the Gaussian VAR processes.
For  MBPs, we are interested in understanding the dependence of each variable on the history of the process. We consider a time series of $N$ Bernoulli variables where each variable depends on at most $p$ lags of the process, resulting in $N^2p$ possible interaction parameters. These $N^2p$ interactions can  be arranged in an $N\times N\times p$ tensor $\Theta$, where $\Theta_{ij\ell}$ captures the effect of variable $j$ from $\ell$ lags ago on variable~$i$.

Collectively, each $N\times N$ \textit{slice} given by $\Theta_{**\ell}$ for $\ell\in\{1,2,\ldots, p\}$ indicates the coupling between all the pairs of variables that are lag-$\ell$ apart. On the other hand, each \textit{fiber} of length $p$ along the third dimension $\Theta_{ij*}$, for $i,j\in\{1,2,\ldots, N\}$, can be thought of as a \textit{filter} that modulates the behavior of variable $i$ on the past responses of variable $j$. 
Delineating the parameters in this way helps with the interpretation of the dynamics of the process. 
For example, in neural signal processing, these filters $\Theta_{ij*}$ encode properties of the spiking behavior of the neurons, as demonstrated by~\cite{weber2017capturing} where they provide a characterization of the filter coefficients that incite activities such as bursting, tonic spiking, phasic spiking and several others. 
Similarly, a slice $\Theta_{**\ell}$ can be thought of as an adjacency matrix for the \textit{lag-$\ell$ influence network}.   If neurons $i$ and $j$ are not connected, then $\Theta_{ij\ell}=0$ for all $\ell$. However, when two neurons are connected, it is possible to have different patterns of influence at different time lags. Hence, the influence networks could potentially be different for each $\ell$.

In order to reveal the structure of these influence networks, one needs to estimate the tensor parameter $\Theta$, given a sample of observations from the process. 
In many scenarios, the parameter tensor $\Theta$ is known to be sparse.
For example, in the neural setting, each neuron, or the functional unit of the cortex, has
limited direct connections to other units.
To exploit the sparsity assumption in forming an estimate of $\Theta$,
we propose and provide a rigorous analysis of an $\ell_1$ penalized  maximum likelihood estimator (MLE).

For Gaussian random variables, the $\ell_1$ penalized MLE takes the form of a regularized least-squares (LS) tensor regression.
However, for the Bernoulli GLM, maximizing the likelihood differs from  ordinary least-squares. Although one can still use a regularized LS estimator, incorporating the Bernoulli likelihood allows one to capture more information, i.e., achieve smaller variance. The resulting $M$-estimator in the Bernoulli case  resembles a regularized logistic regression problem which is in the class of convex problems that can be solved efficiently. 

While traditional statistical methods work under the assumption that the number of available samples significantly exceeds the number of parameters, \ie, $n\gg N^2p,$ the assumption  is often not true in several applications. Going back to the neuroscience application, in most spiking neural data from in vivo measurements, a limited sample of spike trains are available for an experiment, owing to constraints such as subject fatigue, changes in sensor characteristics, and sensitivity to exogenous laboratory conditions. 
In addition, synaptic connectivity is time-varying and hence there is a limited period over which
the model parameters can be assumed constant.

In such a scenario, statistically sound methods are desired that provide guarantees on the fidelity of the model trained over a limited sample of observations, especially with $n\ll N^2p$.
However, in spite of the model identification problem being ill-posed in this case,  it is  often possible to perform reliable estimation even with $ n = \bigO(\log (N^2p))$, by constraining the parameter $\Theta$ to lie in a low-dimensional subspace, with $s$ degrees of freedom, where $s\ll N^2p$. Several such low-dimensional subspaces have been investigated in the literature on compressive sensing and high-dimensional statistics. Examples include the (elementwise) \textit{sparsity}  where approximately $s$ of the $N^2p$ entries are assumed non-zero, or the \textit{low-rank} assumption on the parameter tensor where it is assumed to be the sum of a few rank-1 tensors. These assumptions are often practically valid, enhance the interpretability of the model, and make the problem well-posed when
$n=\mathsf{poly}(s,\log( N^2p ))$. Moreover, the estimation is computationally tractable due to convex optimization based estimators for these, even for the tensor models \cite{raskutti2015convex}.

Our focus in this work is on the ``approximately sparse'' model, as motivated by the network-of-fibers structure  in the neural spike train application. The parameter, $\Theta$, is assumed to be well-approximated by an \mbox{$s$-sparse} tensor. Such a constraint is incorporated by regularizing the (convex) maximum likelihood problem using the elementwise {$\ell_1$ penalty}. Our main result in Theorem~\ref{th:ML_main_result} establishes the consistency of this regularized MLE in the high-dimensional regime of $n =\mathsf{poly}(s,\log (N^2p))$ under some regularity conditions. Despite the fact that the techniques for establishing such high-dimensional guarantees on regularized $M$-estimators are by now fairly mature \cite{buhlmann2011statistics},  significant challenges remain in analyzing dependent non-Gaussian processes.

\subsection{Key contributions} A major theoretical contribution of our work is to establish the so-called restricted strong convexity (RSC)~\cite{negahban2012unified} for the log-likelihood of  a dependent non-Gaussian process. This requires a {restricted eigenvalue} condition for the MBP, which is nontrivial due to the non-Gaussian and highly correlated entries of the resulting design matrix.
What makes the problem more challenging is the existence of feedback from more than just the immediate past (the case $p > 1$). 

We establish the RSC for general $p \ge 1$ using the novel approach of viewing the $p$-block version of the process as a Markov chain. The problem becomes significantly more challenging when going from $p=1$ to {even $p=2$}.
The difficulty with this \textit{higher-order} Markov chain is that its contraction coefficient is trivially 1. We develop techniques to get around this issue which could be of independent interest (cf.~Section~\ref{sec:concentration}). Our techniques hold for all $p\geq 1$.

Much of the previous work towards proving the RSC condition of the log-likelihood function has either focused on the independent sub-Gaussian case~\cite{raskutti2011minimax,zhang2008sparsity} or the dependent Gaussian case~\cite{basu2015regularized,raskutti2015convex} for {which powerful Gaussian concentration results such as the Hanson--Wright inequality~\cite{rudelson2013hanson} are still available}. 
Our approach is to use concentration results for Lipschitz functions of discrete Markov chains, and  strengthen them to uniform results using metric entropy  arguments. In doing so, we circumvent the use of 
	empirical processes which require additional assumptions for MBP estimation~\cite{rakhlin2015sequential}.
	Moreover, our approach allows us to identify key assumptions on the decay rate of the \textit{fiber}s of the parameter,  for sample-efficient estimation.

Although MBP time series are often modeled using the $\mathsf{logit}$ link function, {our analysis allows}
for any Lipschitz continuous, log-convex link function. {The analysis} brings out crucial properties of the link function, and the role it plays in determining the estimation error and sample complexity.

\subsection{Previous work}

{Estimating the parameters of a multivariate time series has {been of } interest in recent years. {Much of the work, however, focuses on Gaussian VAR($p$) processes with linear feedback %
		\cite{basu2015regularized,cai2016estimating,mcmurry2015high,
			mei2017signal,ahelegbey2016sparse}. For these models, a restricted eigenvalue condition can be established fairly easily, by reducing the problem, even in the time-correlated setting, to the concentration of quadratic functionals of Gaussian vectors for which powerful inequalities exist~\cite{rudelson2013hanson}.}
	These techniques do not extend to non-Gaussian setups.}

In the non-Gaussian setting, Hall et al.~\cite{hall2016inference,zhou2018non}  recently considered a multivariate time series evolving as a GLM driven by the history of the process. The autoregressive MBP with $p=1$ lags is a special case of this model. 
They provide statistical guarantees on the error rate for the $\ell_1$ regularized estimator, although their assumptions on the parameter space are restrictive when applied to the MBP. More importantly, their results are restricted to the case $p=1$ which does not allow the explicit encoding of long-term dependencies as observed in crucial neuronal phenomena such as periodic spiking. 

More recently, Mark et al.~\cite{mark2018network,mark2017network} considered a similar model for multivariate Poisson processes with  lags $p>1$, albeit either through predetermined basis functions or by restricting to lags $p=1$ or $p=2$. A key contribution of us is to bring out the explicit dependence on $p$ in multilag MBP models, allowing for a general $p \ge 1$. We show how the scaling of the sample complexity and the error rate with $p$ can be controlled by the properties of the link function and a certain norm of the parameter tensor.

Our results improve upon those in~\cite{hall2016inference,mark2018network} when applied to the MBP. Due to the key observation that a $p$-lag MBP can be viewed as a discrete Markov chain, our analysis relaxes several assumptions made by~\cite{hall2016inference,mark2018network}. In doing so, we achieve better sample complexities with explicit dependence on $p$. Our analysis borrows from martingale-based concentration inequalities for Lipschitz  functions of Markov chains~\cite{kontorovich2008concentration}.

The univariate Bernoulli process for $p\geq 1$ was considered by Kazemipour et. al.~\cite{kazemipour2018compressed,kazemipour2017robust} where they {analyzed} a multilag Bernoulli process for a single {neuron}. Their analysis does not extend to $N>1$ case. Even for $N=1$, their analysis is restricted to the biased process with $\Prob(x^t_i=1|X^{t-1})<\half$ for all $t$.
{Mixing times of the MBP} have been considered in \cite{katselis2018mixing}. However, their discussion is again limited to $p=1$.

\medskip
The rest of the paper is organized as follows. In Section~\ref{sec:formulation}, we present the MBP model and the proposed regularized MLE (R-MLE). Section~\ref{sec:theorem} presents our main result, Theorem~\ref{th:ML_main_result}, on the consistency of the R-MLE and discusses its assumptions and implications. In Section~\ref{sec:simulations}, we provide some simulation results corroborating the theory. Section~\ref{sec:sketch} provides the outline of the proof of Theorem~\ref{th:ML_main_result}. In Section~\ref{sec:concentration}, we present techniques we developed for deriving  concentration inequalities for dependent multivariate processes. Other key technical results needed in the proof are given in Section~\ref{sec:aux_lemmas_proofs}. We conclude with a discussion in Section~\ref{sec:conclusion}. %

\section{Problem Formulation}
\label{sec:formulation}
Consider an $N$-dimensional time series $\set{x^t}_{t=1}^n$, where $t$ denotes time and each $x^t = (x_i^t) \in \Real^N$. %
A general framework for analyzing $\set{x^t}$ is to use a GLM for modeling the conditional mean of the \textit{present} given the \textit{past}, \ie, $p_\Theta(x^t_i\mid X^{t-1}) $ is such that
\[\Exp[x_i^t\mid X^{t-1}] = f\big(\inner{\Theta_{i\mdot\mdot},X^{t-1}}\big)\]
independently across coordinates  $i\in [N]$.
Here, $X^{t-1}=[x^{t-1}\ x^{t-2}\ \cdots\ x^{t-p}]\in\Real^{N\times p}$ is the ``$p$-lag'' history at time $t$, and $\Theta_{i\mdot \mdot}\in \Real^{N\times p}$ is the $i^{\rm th}$ slice of the parameter tensor $\Theta\in \Omega \subseteq \Real^{N\times N\times p}$
along the first dimension. $f$ is the inverse link function. The notation $\inner{\cdot,\cdot}$ denotes the usual Euclidean inner product between two matrices, i.e., $\ip{ \Theta_{i\mdot\mdot},X^{t-1}}$ equals
\[ 
\sum_{j,\ell} \Theta_{ij\ell} X^{t-1}_{j\ell} = \sum_{j,\ell} \Theta_{ij\ell} \,x^{t-\ell}_{j}  = \sum_{\ell} \ip{\Theta_{i*\ell}, x^{t-\ell}} 
\]
where we note the useful identity $X_{j\ell}^{t-1} = x^{t-\ell}_j$. (The inner product in the last equality is the usual vector inner product.)
The entry $\Theta_{ij\ell}$ captures how much the $j^{\rm th}$ dimension of the process at time lag $\ell$ affects the distribution of the $i^{\rm th}$ dimension, at each time instant. The model relates the distribution of $x^t$ to the history of the process over $p$ time lags $X^{t-1}$.

In other words, for the MBP model, we have
\begin{align}\label{eq:main_GLM}
\begin{split}
x_i^t \mid X^{t-1} &\sim {\rm Ber}\left(z_i^{t}\right),\\
z_i^t := z_i^t(\Theta) &:= f\left(\inner{\Theta_{i \mdot \mdot},X^{t-1}}\right),
\end{split}
\end{align}
where %
$f:\Real\rightarrow[\eps,1-\eps]$ for some $\eps\in(0,\half)$, is the inverse link function.
Note that  $z_i^t = \Prob(x_i^t=1\mid X^{t-1})$ represents the conditional probability of spiking for neuron $i$ at time $t$, given the $p$ lag history of the ensemble.  

A general (time-invariant) time series $\{x^t\}$ over the discrete space $\bin^N$ with { finite dependence on the past, i.e., }
\[
\Prob\big(x^t = a \mid x^{t-1},x^{t-2},\dots\, \big) = 
\Prob\big(x^t = a\mid x^{t-1},x^{t-2}\ldots, x^{t-p} \,\big), \quad a \in \{0,1\}^N
\]
can be modeled as a homogeneous Markov chain over $\{0,1\}^{Np}$, with possibly $\bigO(2^{Np})$ parameters in the worst case. Hence the MBP model is a $N^2p$ dimensional representation of a subset of these models. The representation power of the MBP in \eqref{eq:main_GLM} is beyond the scope of this paper.

We are interested in estimating the ``true'' parameter tensor $\Theta$ from $n$ samples $\{x^t\}_{-p+1}^{n}$  of a process. (We assume that the first $p$ samples corresponding to $i=-p+1$ to $i=0$ are given for ``free'' since during the estimation the $n$ data points $\{x^t\}_{t=1}^n$ are \emph{regressed} onto their history $\{X^{t-1}\}_{t=1}^n$ via the likelihood function of the Bernoulli distribution.) %

Although the parameter space $\Omega$ has ambient dimension $N^2p$, in real-world applications, $\Theta$ often resides in or is well-approximated by a low dimensional subspace of $\Omega$, allowing reliable estimation even if $n\ll N^2p,$ as is desired in several applications. In the literature on high-dimensional statistics, this is often the assumption and several such low dimensional subspaces are now well-studied and ubiquitous in analyses, some examples include tensors being low-rank, exactly sparse, approximately sparse and sparse with a specific structure. Here we focus on parameter estimation for the approximately sparse parameter model where the true parameter has a ``good enough'' sparse approximation.

Let the process be generated by a true parameter $\Theta^*$. We assume $\Ths$ to be approximately $s$-sparse. More precisely, we assume that the following quantity
\begin{align}\label{eq:compressability}
\sigma_s(\Theta^*):=\min_{|S|\leq s}\tonorm{ \Theta^*_{S^c}}, %
\end{align}
decays fast as a function of $s$, where $\Theta^*_{S^c}$ is the tensor $\Theta^*$ with support restricted to $S^c$, the complement of $S\subseteq [N]^2\times[p]$.
This quantity $\sigma_s(\Theta^*)$ 
captures the $\ell_{1,1,1}$ approximation error when $\Theta^*$ is approximated by an $s$-sparse tensor. For an exactly $s$-sparse tensor $\Theta^*$, we have $\sigma_s(\Theta^*) = 0$. In general, we do not impose any constraint on $\sigma_s(\Theta^*)$ and state a general result involving this parameter. We denote by $S^*$ the optimum set that solves \eqref{eq:compressability}.
For future reference, we also define
\begin{align}\label{tau:def}
\tau_s^2(\Theta^*) :=\frac{\sigma_s^2(\Theta^*)}s.
\end{align}

\paragraph{Notation.} Here and in what follows $\norm{\Theta}_{p,q,r}$ denotes the norm on tensors obtained by collapsing the dimensions from right to left by applying $\ell_r$, $\ell_q$ and $\ell_p$ norms in that sequence. Thus $\norm{\cdot}_{1,1,1}$ is the elementwise $\ell_1$ norm of the tensor, i.e., the sum of the absolute values of all its entries, and $\norm{\cdot}_{\infty,\infty,\infty}$ is the absolute value of the entry of the tensor with largest magnitude, which is the dual norm for $\ell_{1,1,1}$ norm. $\|\Theta\|_{0}$ denotes the number of non-zero entries in $\Theta$.

\section{Main Result}
\label{sec:theorem}

We study a regularized maximum likelihood estimator (R-MLE) for $\Theta^*$. The (normalized)  negative log-likelihood of the Bernoulli process~\eqref{eq:main_GLM} is given by,
\begin{align}\label{eq:Likelihood}
\like(\Theta) %
&= -\frac{1}{n}\sum_{t=1}^n \sum_{i=1}^N \ell_{it}\left(\inner{\Theta_{i**},X^{t-1}}\right)%
\end{align}
where %
$\ell_{it}(u):=x_i^t\log(f(u))+(1-x_i^t)\log(1-f(u))$.
In order 
to incorporate the sparse approximability of $\Theta^*$
during estimation, we penalize the likelihood via an element-wise $\ell_1$ regularization,
leading to the R-MLE  %
\beq
\widehat\Theta = \underset{\Theta \, \in \, \Real^{N\times N\times p}}{\rm argmin}\quad \like(\Theta)+\lambda_n \norm{\Theta}_{1,1,1},\label{eq:MLestimator}
\eeq
where $ \norm{\Theta}_{1,1,1}$ is the element-wise $\ell_1$ norm of the 3-tensor $\Theta$ encouraging a sparse solution to the optimization problem. 
Consider $n$ samples of the multivariate Bernoulli process generated by equation \eqref{eq:main_GLM} with parameter $\Theta=\Theta^*$. We denote these samples by $\mbb X:=\{x^i\}_{i=-p+1}^n$.
We further assume that the process satisfies the the following regularity conditions:

\begin{enumerate}[label=(A\arabic*)]%
	\item \label{ass1} The samples $\mbb X=\{x^t\}_{t=-p+1}^n$ are drawn from a stable and wide-sense stationary process which has a power spectral density matrix, %
	\begin{align}\label{eq:power:spec:def}
	\mc X(\omega):=\sum_{\ell=-\infty}^\infty\Cov(x^t,x^{t+\ell}) e^{-j\omega\ell}\in\mbb C^{N\times N},
	\end{align}
	for $\omega\in[-\pi,\pi)$ {with minimum eigenvalues uniformly bounded below as }
	\begin{align}\label{eq:lamin:psd}
	\underset{\omega\,\in\,[-\pi,\pi)}\min\,  \lambda_{\min}(\mc X(\omega)) = c_\ell^2 > 0.
	\end{align}
	\item \label{ass2} The nonlinearity or inverse link function $f:\Real\rightarrow[\eps,1-\eps],$ for some $\eps\in(0,\half)$ is twice differentiable and has a Lipschitz constant $L_f$, \ie, $|f(u)-f(v)|\leq L_f |u-v|$. We also assume that both $-\log f$ and $-\log (1-f)$ are strongly convex with curvatures  bounded below by $c_f > 0$. %

\end{enumerate}

The following quantities will be key in stating the error bounds:
\begin{align}
g_f(\Theta^*) &:= \bigg(\frac{3L_f^2}{2\eps}\sum_{\ell=1}^p\sum_{i=1}^N
\Big(\sum_{j=1}^N\sum_{k=\ell}^p|\Theta^*_{ijk}|\Big)^2\bigg)^{1/2}, \,\label{eq:definition_of_g}%
\\
\Gf &:= 8c_f^2\Big[ 1 + \frac{p^2}{\big(\frac{1}{g_f(\Theta^*)}-1\big)^2} \Big].\label{eq:definition_of_G}
\end{align}
{We note that $g_f(\Theta^*)$ is a valid norm on the space of 3-tensors. This quantity captures how fast the process is mixing, and controls how fast Lipschitz functions of the process concentrate. It indirectly controls the hardness of the estimation problem. See Section~\ref{sec:concentration} for the details. } We are now ready to state our main result: %
\begin{theorem}	\label{th:ML_main_result}
	Let $\{x^t\}_{t=-p+1}^n$ be a process generated by \eqref{eq:main_GLM}
	with parameter $\Theta^*$ %
	and satisfying~\ref{ass1} and \ref{ass2}. {Suppose $g_f(\Theta^*)<1$.}
	Then there exist positive universal constants $c, c_1$ and $c_2$ such that for 
	\begin{align*}
	n   \ge c_1 \frac{ G_f(\Theta^*)}{c_f^2 \,c_\ell^6}\, s^3    \log (N^2p),
	\end{align*}
	any solution $\wh\Theta$ to \eqref{eq:MLestimator} with $\lambda_n = c_2\frac{L_f}{\eps} \sqrt{\frac{\log (N^2p)}n}$, satisfies
	\begin{align}\label{eq:err:bound}
	\norm{\wh\Theta-&\Theta^*}_F^2 
	\leq C \Big[ \frac{s\log(N^2p)}{n} + \taut(\Theta^*)\sqrt{\frac{\log (N^2p)}{n}} \Big],
	\end{align}
	with probability at least $1-n^{-c} - 2(N^2p)^{ -c_0 s}$,
	where $\taut_s(\Theta^*) :=\tau_s^2(\Theta^*)+\sigma_s(\Theta^*)$. 
	The constant $c_0 = \bigO(c_\ell^{-2})$ only depends on $c_\ell$ and $C = \bigO( \max\big\{\frac{L_f}{\eps c_f c_\ell^2},1 \big\}^2)$ only depends on the stated constants.
\end{theorem}

Theorem~\ref{th:ML_main_result} is proved in Section \ref{sec:sketch}. We first discuss important implications of this result, and the conditions imposed on the scaling of various parameters. %

\subsection{Remarks on Theorem \ref{th:ML_main_result}}\label{sec:remarks}
The two terms in the error bound~\eqref{eq:err:bound} correspond to the estimation and approximation errors, respectively. The estimation error, in general, scales at the so-called \textit{fast rate} $s\log(N^2p)/{n}$ in our setting, while the approximation error scales with the slower rate $\taut^2(\Theta^*)\sqrt{\log (N^2p)/{n}}$.
For the exact sparsity model, where $\sigma_s(\Theta^*)=0,$ the approximation error vanishes (since $\wt\tau^2(\Theta^*)=0$) and we achieve the fast rate.

The overall sample complexity for consistent estimation (ignoring constants) is thus
\begin{align}\label{eq:samp:comp}
n \gtrsim \Gf \max\{s^3, \taut^4_s(\Theta^*) \}\,\log(N^2p) \,.
\end{align}
\paragraph{Scaling with $s$.} According to~\eqref{eq:samp:comp}, the scaling of $n$ in the sparsity parameter ``$s$'' is at best $O(s^3)$, corresponding to the case of hard sparsity where $\taut^4_s(\Theta^*) = 0$. %
While an $O(s^3)$ dependence is not ideal, it is not clear if it can be improved significantly without imposing restrictive assumptions. {It is clear that one cannot do better than $O(s)$, the optimal scaling in the linear independent settings. In our proof, the additional $s^2$ factor comes from  concentration inequality~\eqref{eq:main_concentration} in Proposition~\ref{prop:concentration_via_martingale1}. There, if one were able to show sub-Gaussian concentration for deviations of the order of $\frob{\Delta}^2$ instead of $\norm{\Delta}_{2,1,1}^2,$ then the additional $s^2$ can be removed.
	It remains open whether such concentration is possible and under what additional assumptions. Figure~\ref{fig:sparsity_sample_error_grid} in Section~\ref{sec:simulations} suggests a superlinear dependence on $s$, hinting that the situation may not be as simple as the i.i.d. case.}

In comparison, for $p\leq 2$, a sample complexity of $\rho^3 \log(N)$ was reported in~\cite[Cor. 1]{hall2016inference}, whereas~\cite[Thm 4.4]{mark2018network} requires $U^4 s \log(N)$ samples where $\rho$ and $U$ are parameters defined in their respective models, both of which can potentially grow as $\Omega(s)$ unless assumed otherwise.

\paragraph{Scaling with $p$.} {For $N=1$, our result is the first to provide a sample complexity logarithmic in $p$ which holds for all $N$. In contrast,~\cite[Thm. 1]{kazemipour2017robust} requires $s^{2/3} p^{2/3} \log(p)$ samples and relies heavily on $N=1$.}

\begin{table}[t]
	\caption{\label{tab:pscale}Scaling of $\Gf$ with $p$. Here $\alpha>3/2$, $\beta>0$.} %
	\centering
	\renewcommand\arraystretch{1.2}
	\begin{tabular}{|l|*{3}{c|}}\hline
		\backslashbox[35mm]{$|\Theta_{ij\ell}|$}{$L_f$}
		&\makebox[3em]{$\bigO(1)$}&\makebox[3em]{$\bigO(p^{-1})$}&\makebox[3em]{$\bigO(p^{-2})$}\\\hline
		$\bigO(1)$ &$\bigO(p^5)$&$\bigO(p^3)$&$\bigO(1)$\\\hline
		$\bigO(\ell^{-\alpha})$ or $\bigO(e^{-\beta \ell})$ &$\bigO(p^2)$&$\bigO(1)$&$\bigO(1)$\\\hline
	\end{tabular}
\end{table}

Our bound scales with $p$ through the quantity $\Gf$. The scaling depends on the behavior of the tail of $\ell \mapsto |\Theta_{ij\ell}|$, that is, how fast the ``\textit{influence from the past}'' dies down. For different regimes of the influence decay, the scaling of $\Gf$ is summarized in Table~\ref{tab:pscale}.
In the worst case, without any assumptions on $\Theta$ and with $L_f=\bigO(1)$, $\Gf$ could scale as $p^5$. {Although this is not ideal, it is analogous to constant $U^4$ in the sample complexity of~\cite{mark2018network}, which is derived under more assumptions on $\Theta^*$. (The dependence  of $U$ on $p$ in that result is also somewhat complicated.) 
	
	On the other hand, under mild assumptions of polynomial or exponential tail decay, the dependence on $p$ is much better: If  $|\Theta_{ij\ell}|$ decays polynomially in lag $\ell$, i.e., $|\Theta_{ij\ell}|=\bigO(\ell^{-\alpha})$, uniformly in $i,j$,  for any $\alpha>3/2$,  or exponentially  $|\Theta_{ij\ell}| = \bigO(e^{-\beta\ell})$ for $\beta > 0$, then the sample complexity reduces by a factor of $p^3$.}
{Furthermore, if the  Lipschitz constant $L_f$ is allowed to drop with $p$, more reduction in $G_f(\Ths)$ and hence the sample complexity is possible, as illustrated in Table~\ref{tab:pscale}.}

Appendix~\ref{sec:scaling_of_G} provides derivations for Table~\ref{tab:pscale}. Better bounds on $G_f$ and hence the sample complexity can be obtained by imposing suitable structural assumptions on $\Theta^*.$

\paragraph{Scaling with $N$.} Our results  have a logarithmic dependence on $N$, the number of neurons in the context of neuronal ensembles, which is a notable feature of our work. We overcome the $N>1$ barrier for the MBP model, while also allowing $p\geq 1$.

{\paragraph{Assumptions.} We use Assumption~\ref{ass1} to guarantee  that restricted strong convexity (RSC)  holds at the population level. The RSC is key in guaranteeing } that any parameter tensor $\wh \Theta$ that maximizes the regularized likelihood does not deviate far from the true parameter. For the $N=1$ case, it implies that the process does not have zeros on the unit circle in the spectral domain. %
{Assumption \ref{ass1} is by now standard in estimating time-series~\cite{raskutti2015convex,basu2015regularized}. It relates to the ``flatness'' of the power spectral density (PSD)~\cite{basu2015regularized}. Controlling the scaling of $c_\ell$ in terms of $\Theta$ is a non-trivial research question, even for Gaussian AR($p$) processes, since the relation between the PSD and the parameter is via the Z-transform. While there could be pathological $\Theta$ for which $c_\ell=o(1)$, the set of parameters for which $c_\ell=\Omega(1)$, is largely believed to be non-trivial. A line of work~\cite{basu2015regularized,hall2016inference,mark2018network} obtains weak bounds on $c_\ell$ which could decrease with $s$, hence the authors need to further assume $s=O(1)$.}

Assumption~\ref{ass2} is not too restrictive. For example, for the $\texttt{logit}$ link, %
$f(u)=f_\alpha(u)=1/(1+e^{\alpha u})$, the sigmoid function, we have %
\begin{align*}
c_f &= \min_{ u } \;\alpha^2f_\alpha(u)(1-f_\alpha(u)) = \alpha^2 \eps(1-\eps)
\end{align*}
assuming that $f \in [\eps,1-\eps]$, that is, $\alpha |u| \le \log\frac{1-\eps}\eps$. The Lipschitz constant in this case satisfies $L_f\leq \frac{\alpha}{4}.$

\section{Simulations}\label{sec:simulations}
We evaluate the performance of the estimator in \eqref{eq:MLestimator} using simulated data. We use two different metrics of performance: (1) the estimation error in Frobenius norm, and (2) support recovery, i.e., assuming that the true parameter tensor is exactly $s$-sparse, how does the support estimated from $\ThML$ compare to the support of $\Th^*$. To do so, we need to estimate the support from $\ThML$. If we know the sparsity, we can estimate the support by taking the indices corresponding to the $s$ largest entries of $\ThML$ in magnitude. If we do not know the sparsity in advance, we can estimate the support based on a threshold chosen by cross-validation. Given a threshold $\gamma$, the estimated support would be
\[\widehat{\supp}(\Theta) := \set{(j, k , \ell): |{ {\widehat{\Theta}}_{{\rm ML}_{jk\ell}} }|\geq \gamma}.\label{eq:support_estimation}
\]
Note that our theoretical results do not give any guarantees for support recovery. In order to guarantee support recovery, a stronger result bounding the error uniformly for each entry of $\ThML$ is required, i.e., we need to control
$\norm{\widehat{\Theta}_{\rm ML}-\Th^*}_{\infty,\infty,\infty}$
with high probability.  Therefore, more work is needed to obtain theoretical guarantees for support recovery. Nevertheless, our simulations show that the estimator is able to recover the support very well.

We first simulate a network with $N = 20$ units and $p = 20$ lags, giving  a parameter space of dimension $N\times N \times p = 8000$. We generate a uniformly random sparsity pattern over this space, and then generate the data using the parameter tensor, according to model~\eqref{eq:main_GLM}. We use the sigmoid non-linearity $f(t) = 1/(1+e^{-t})$.
Next we use this data to obtain $\ThML$ and compute the estimation error in Frobenius norm $\frob{\ThML - \Theta^*}$. %
This process is repeated for $20$ independent runs. The regularization parameter was set to $\lambda_n=100\sqrt{\log(N^2p)/n}$.

\begin{figure*}[t!]
	\centering
	\begin{subfigure}[b]{0.495\textwidth}
		\centering
		\includegraphics[width=\textwidth]{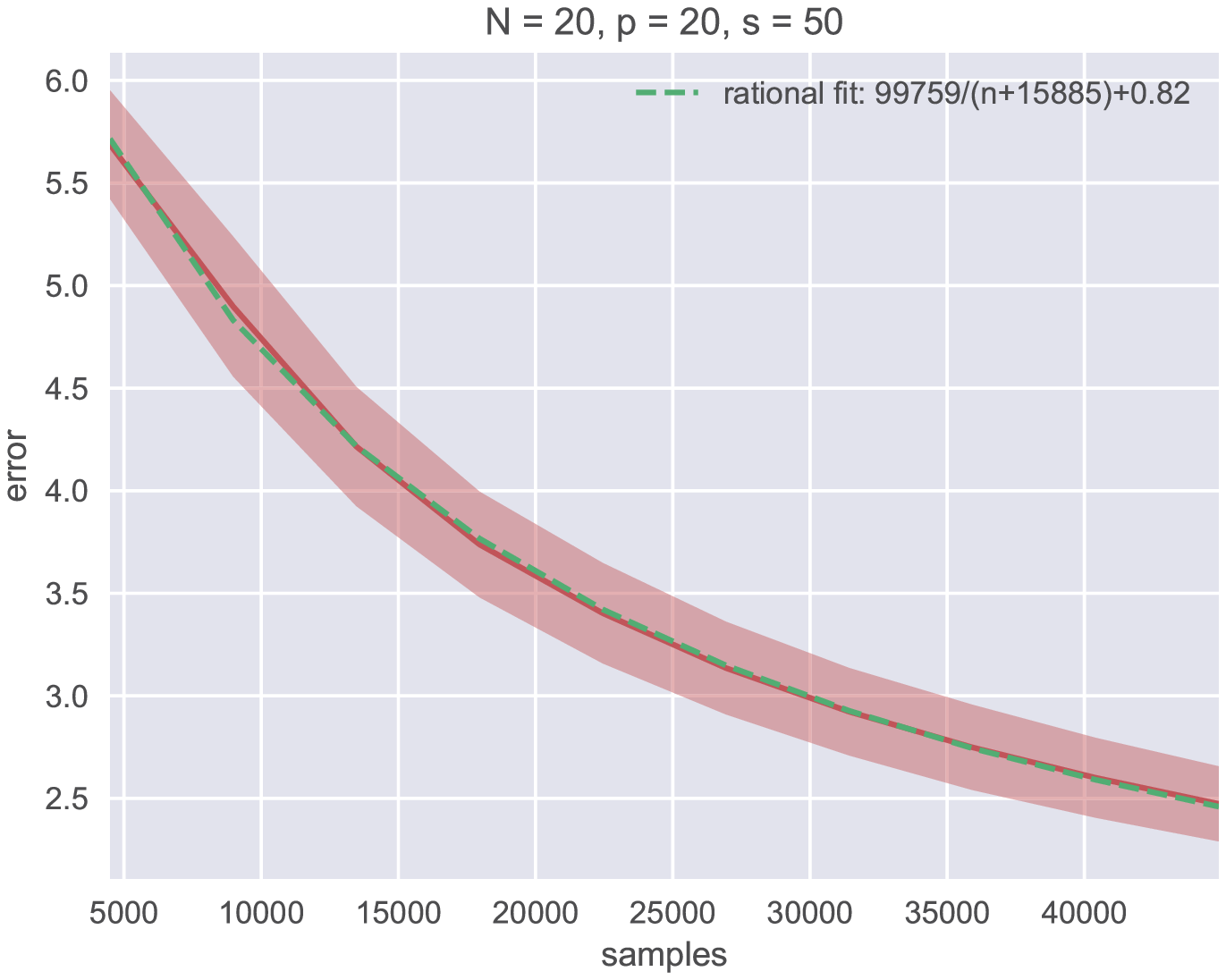}
		\caption{Error vs.\ sample size for sparsity $s=50$.}   
		\label{fig:err_vs_Sample}
	\end{subfigure}
	\hfill
	\begin{subfigure}[b]{0.495\textwidth}  
		\centering 
		\includegraphics[width=\textwidth]{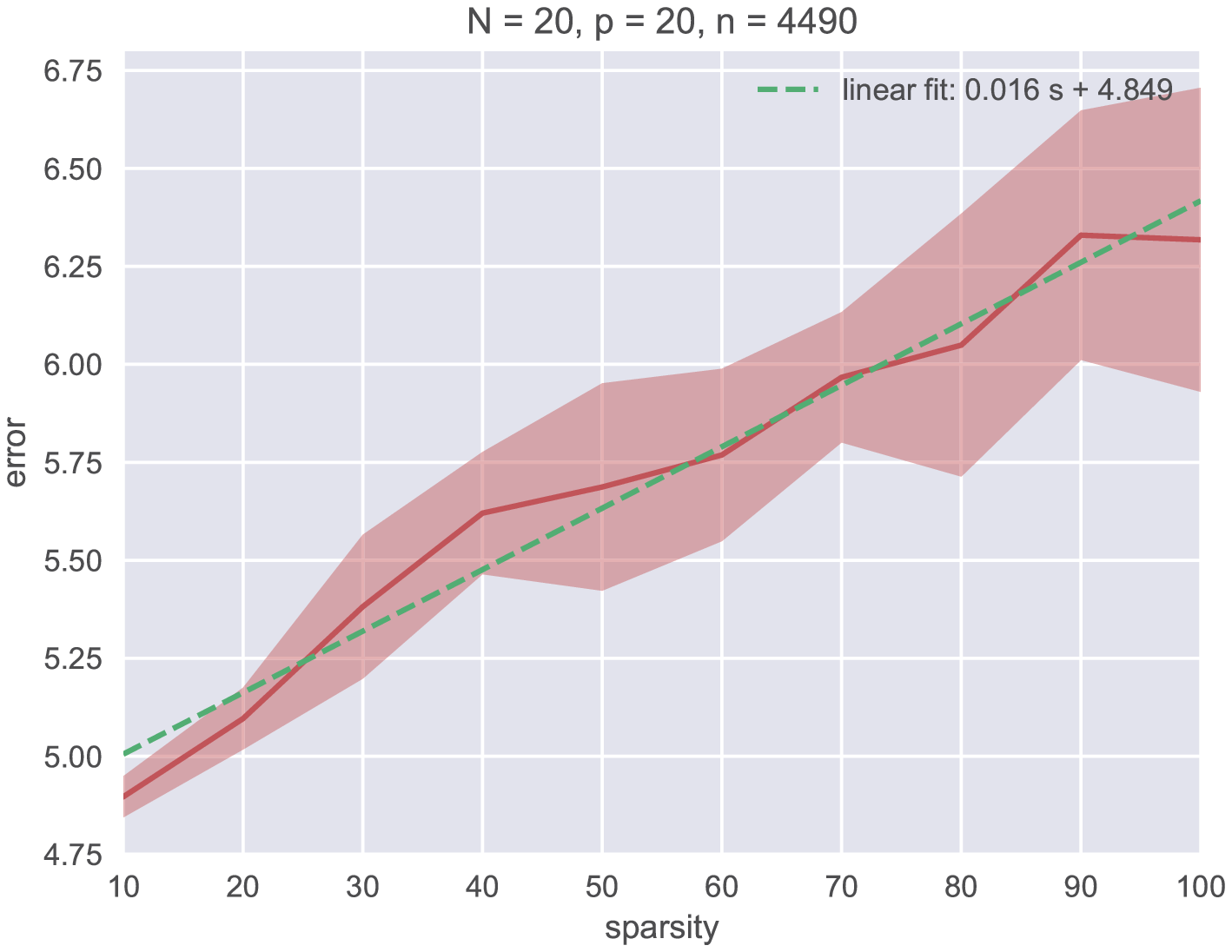}
		\caption{Error vs.\ sparsity for sample size $n=4490$}  
		\label{fig:err_vs_sparsity}
	\end{subfigure}

	\begin{subfigure}[b]{0.475\textwidth}   
		\centering 
		\includegraphics[width=\textwidth]{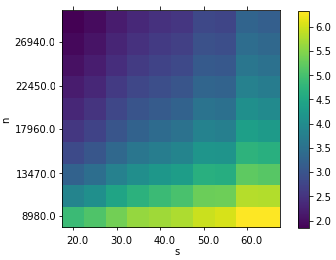}
		\caption{Average Frobenius norm of the error over $20$ runs for a network with $N = 20$ units and $p=20$ lags. Each pixel corresponds to a sample size $n$ and a sparsity level $s$ for $\Theta^*$. Darker colors indicate smaller estimation error.}   
		\label{fig:sparsity_sample_error_grid}
	\end{subfigure}
	\quad
	\begin{subfigure}[b]{0.475\textwidth}   
		\centering 
		\includegraphics[width=\textwidth]{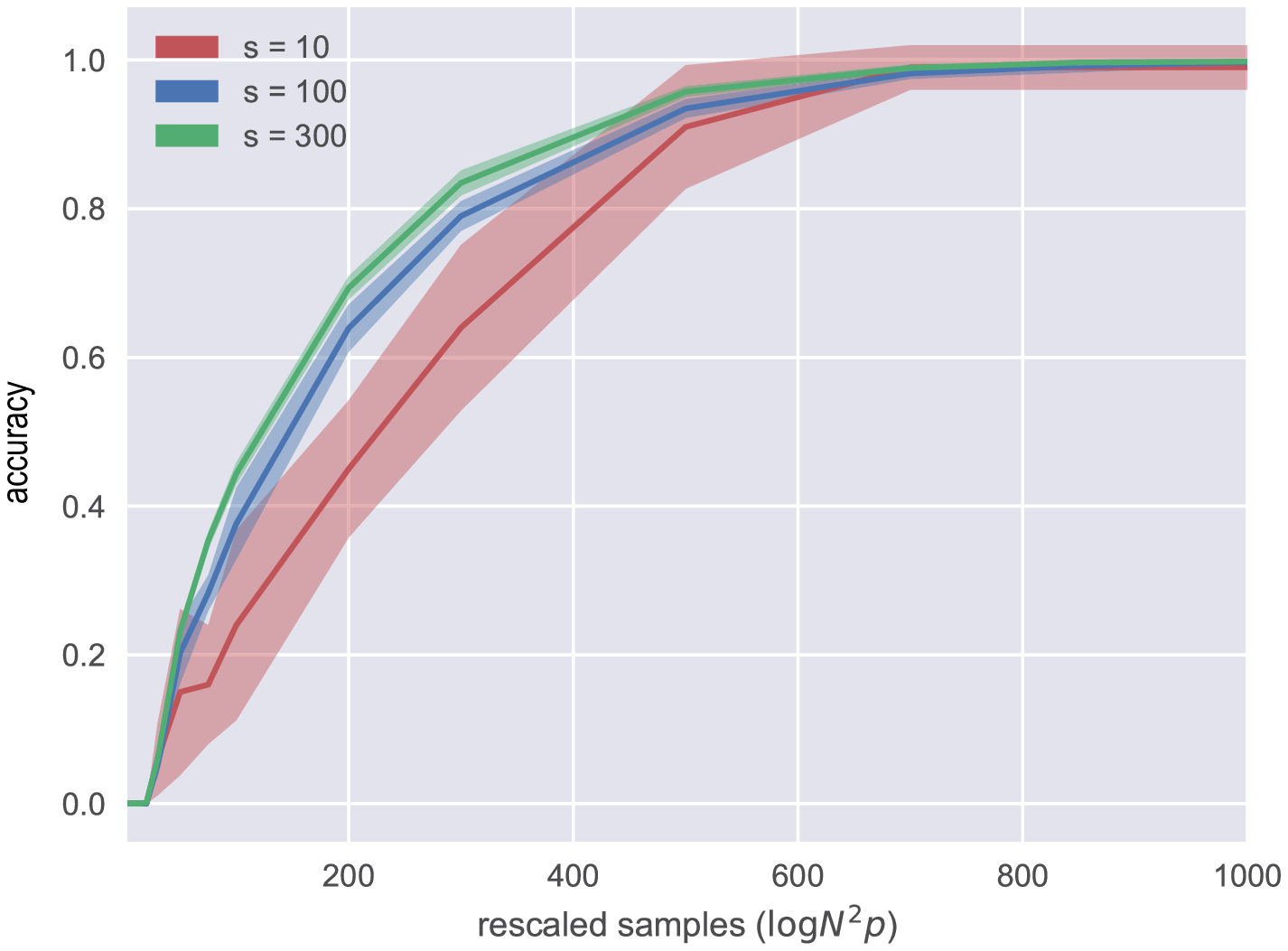}
		\caption{Fraction of support size recovered by taking the largest $s$ entries of $\hat{\Theta}$ as the estimator of support, for a model with $N=100$, $p=1$, and three different sparsities. The samples are rescaled by $\log N^2p$. %
		}   
		\label{fig:support_recovery}
	\end{subfigure}
	\caption[ The average and standard deviation of critical parameters ]
	{\small Simulation results.} 
	\label{fig:error_trends}
\end{figure*}

The results are shown in Figure \ref{fig:error_trends}. Figure \ref{fig:err_vs_Sample} shows how the error changes with the sample size. The shaded area represents one standard deviation of the estimation error over 20 runs. For comparison, a rational function fit of the form $a+ b/(n+c)$ is also plotted on top of the error. In Figure~\ref{fig:err_vs_sparsity}, we plot the error vs.\ sparsity for a fixed sample size. The error grows almost linearly with sparsity. A linear fit of the error is also shown for comparison.
Figure \ref{fig:sparsity_sample_error_grid} shows the average error, over the runs, for different sparsity levels and different sample sizes. As expected, the error goes down as $\Theta^*$ becomes sparser, or sample size $n$ increases.

Finally, the support recovery  performance is shown in Figure~\ref{fig:support_recovery}. In this experiment, 
the network has $N =100$ units but only $p=1$ lag for $s=10, 100, 300$. For recovering the support, we assumed that the sparsity $s$ is known, and took the indices corresponding to the $s$ largest entries of $\ThML$ as the recovered support. The fraction of the correctly recovered indices is plotted against the sample size. Figure~\ref{fig:support_recovery} shows that if the sample size is below some threshold, no entries of the support are recovered, while above the threshold, the recovered fraction gradually increases to $1$.

\section{Proof of the Main Result}\label{sec:sketch}
We now outline the proof of Theorem~\ref{th:ML_main_result}.
The main challenge in the regime  $n\ll N^2p$, is that the empirical Hessian $\nabla^2\like$ is rank-deficient and hence the likelihood cannot be strongly convex (i.e., have a positive curvature) in all directions around the true parameter $\Theta^*$. This means that even though a candidate solution $\wh\Theta$ is a stationary point of \eqref{eq:MLestimator}, it could be far away from $\Theta^*$ if the error vector $\wh \Delta\defeq\wh\Theta-\Theta^*$ lies in the null-space of $\nabla^2\like$. However, for sufficiently large values of $\lambda_n$, one can guarantee that this almost never happens under certain assumptions.

Specifically, if the regularization parameter $\lambda_n$ is large enough, then $\wh \Delta$ lies in a small ``cone-like'' subset of $\Real^{N\times N\times p}$. Now, it suffices to require that $\like$ is strongly convex only over this subset (i.e., $\nabla^2\like$ is uniformly quadratically bounded from below on this set). This observation is by now standard in the high-dimensional analysis of M-estimators. %
 It was shown in~\cite[Lem. 1]{negahban2012unified} that if for any loss function $\mc L$, problem \eqref{eq:MLestimator} is solved with regularization parameter satisfying
 \begin{align}
 	\label{eq:regularizationparam}
 	\lambda_n \geq 2\norm{\nabla\like(\Theta^*)}_{\infty,\infty,\infty},
 \end{align}
 then  for any $S\subseteq [N]\times[N]\times [p]$, the error vector $\wh \Delta$ belongs to the cone-like set $\mbb C( S;\Theta^*)$, defined as
 \begin{align}\label{eq:cone:like:set}
\mbb C( S;\Theta^*)=\left\{\Delta\in\Omega :  \norm{\Delta_{S^c}}_{1,1,1}
 		\leq 3\norm{\Delta_{S}}_{1,1,1}
 		+ 4\norm{\Theta^*_{S^c}}_{1,1,1}\right\}.
\end{align}

In other words $\wh \Delta\in\cap\{ \mbb C(S;\Theta^*)\mid S\subseteq[N]\times [N]\times [p]\}$ if $\lambda_n$ satisfies \eqref{eq:regularizationparam}.

 \begin{definition}[Restricted Strong Convexity]
 	\label{defn:rsc}
For a loss function $\mc L$, let 
 \begin{align}\label{eq:rem:def}
 	\rem \like(\Delta;\Theta^*) := \like(\Theta^*+\Delta)-\like(\Theta^*)-\inner{\nabla\like(\Theta^*),\Delta}
 \end{align}
 be the remainder of the first-order Taylor expansion of the loss function around $\Theta^*$.  %
 	A loss function $\like$ satisfies restricted strong convexity (RSC) relative to $\Theta^*$ with curvature $\kappa > 0$ and tolerance $\tau^2$  if 
 	\begin{align}
 		\label{eq:rsc_first}
 		\rem \like(\Delta;\Theta^*) \geq \kappa \fnorm{\Delta}^2-\tau^2
 	\end{align}
 	for any $\Delta \in \mbb C(S;\Theta^*)$,
where $S$ is any subset of $[N]\times [N]\times [p]$. %

 \end{definition}

The following result gives the desired error bound in Theorem \ref{th:ML_main_result}.
 
\begin{proposition}{\rm (Adapted from 
\cite[Thm. 1]{negahban2012unified})} 
If \eqref{eq:MLestimator} is solved with regularization parameter satisfying
\eqref{eq:regularizationparam}, and if for any $S\subseteq [N]\times [N]\times [p]$ the RSC condition \eqref{eq:rsc_first} holds for all $\Delta\in\mbb C(S;\Theta^*)$, defined in \eqref{eq:cone:like:set}, then we have
 \begin{align}\label{eq:gen:error:bnd}	
 \fnorm{\wh\Theta-\Theta^*}^2 \leq 9\frac{\lambda_n^2 |S|}{\kappa^2}+\frac{\lambda_n}{\kappa}\left(2\tau^2+ 4\norm{\Theta^*_{S^c}}_{1}\right).
 \end{align}
 \end{proposition}

{To apply this result, we first show in Lemma~\ref{lem:gradient} that taking $\lambda_n=\bigO(\sqrt{\log(N^2p)/n})$ is enough for~\eqref{eq:regularizationparam} to hold with high probability.
Equation~\eqref{eq:gen:error:bnd} then gives a family of bounds, one for each choice of $S$. Decreasing $|S|$ potentially increases $\norm{\Theta_{S^c}}_{1}:=\norm{\Theta_{S^c}}_{1,1,1}$ and hence presents a trade-off. We choose an $S$ that balances all the terms in the bound. %
Specifically, we choose $S=S^*$ that solves \eqref{eq:compressability}, with $|S^*|=s$ and $\norm{\Theta^*_{S^{*c}}}_{1,1,1}=\sigma_s(\Theta^*)$. For this choice $S^*$, we show in Proposition~\ref{prop:rsc:main} that \eqref{eq:rsc_first} holds over $\mbb C(S^*;\Theta^*)$ with high probability for $\kappa=\Omega(1)$, and $\tau^2=\sigma_s^2(\Theta^*)/s$. Putting these together proves  Theorem~\ref{th:ML_main_result}.} %
\subsection{Choice of the regularization parameter}

To set $\lambda_n$ such that~\eqref{eq:regularizationparam} holds, we need to find an upper bound on $\norm{\nabla\mc L(\Theta^*)}_{\infty,\infty,\infty}$. Since $\lambda_n$ affects the error bound directly, we would like to choose the smallest $\lambda_n$ that satisfies~\eqref{eq:regularizationparam}. In general, one would like to have a vanishing $\lambda_n$ (as $n \to \infty$) to guarantee consistency. Our next result provides the necessary bound on the gradient of the loss, leading to a suitable choice for the regularization parameter $\lambda_n$:

\begin{lemma}\label{lem:gradient}
For any constant $c_1 > 2$,
\beq
\norm{\nabla \like(\Theta^*)}_{\infty,\infty,\infty} \leq {\frac{L_f}{\eps}}\sqrt{\frac{c_1\log (N^2p)}{n}} \label{eq:grad bound}
\eeq
with probability at least {$1- (N^2p)^{-c}$, where $c = c_1/2-1$.} %
\end{lemma}
\begin{proof}
	Fix $i,j \in [N]$ and $\ell \in [p]$. From~\eqref{eq:Likelihood},
\[
	\frac{\partial \like(\Theta)}{\partial \Theta_{ij\ell}} = - \frac1n \sum_{t=1}^n \ell'_{it}( \ip{\Theta_{i\mdot\mdot}, X^{t-1}}) \,X_{j\ell}^{t-1}
\]	
where
$
 \ell'_{it}(u):= \big[\frac{x_i^t}{f(u)} - \frac{1-x_i^t}{1-f(u)}\big] f'(u).
$
It follows that
\[
	\frac{\partial \like(\Theta^*)}{\partial \Theta_{ij\ell}} 
	= \frac1n \sum_{t=1}^n D^t_{ij\ell} \quad \text{where} \quad  D^t_{ij\ell} := \left(\frac{1-x_{i}^t}{1-z_{i}^t} -\frac{x^t_{i}}{z_{i}^t} \right) f'\big(\ip{\Theta^*_{i\mdot\mdot}, X^{t-1}}\big) \,X_{j\ell}^{t-1}.
\]
Let $\mathcal F^{t-1} = \sigma(x^{t-1},x^{t-2},\dots)$ be the $\sigma$-field generated by the past observations of the process. We have $\ex[ x_i^t \mid \mathcal F^{t-1}] = z_i^t \in \mathcal F^{t-1}$, hence
\begin{align*}
	\ex[ D_{ij\ell}^t \mid \mathcal F^{t-1}] = 
	 \left(\frac{1-\ex[ x_i^t \mid \mathcal F^{t-1}]}{1-z_{i}^t} -\frac{\ex[ x_i^t \mid \mathcal F^{t-1}]}{z_{i}^t} \right) f'\big(\ip{\Theta^*_{i\mdot\mdot}, X^{t-1}}\big) \,X_{j\ell}^{t-1} = 0.
\end{align*}
That is, $\{D_{ij\ell}^t\}_t$ is a  \emph{martingale difference sequence}. Recall that $X^{t-1}_{j\ell} = x_j^{t-\ell} \in \{0,1\}$, while $\|f'\|_\infty \le L_f$ and $z_i^t \in [\eps,1-\eps]$ both by Assumption~\ref{ass2}. If follows that $\{D_{ij\ell}^t\}_t$ is also bounded, i.e., $|D_{ij\ell}^t| \le L_f/\eps$. By the Azuma--Hoeffding inequality for martingale differences~\cite{van2002hoeffding}, 
\[
	\Prob\left(	\Big|\frac{\partial \like(\Theta^*)}{\partial \Theta_{ij\ell}}\Big| >t\right)=\Prob\left(\Big|\frac1n \sum_{i=1}^n D_{ij\ell}^t \Big| > t \right) \leq 2\exp\Big({-}\frac{n\eps^2t^2}{2L_f^2}\Big), \quad t > 0.
\]
Writing $\infnorm{\nabla \like(\Theta^*)} := \norm{\nabla \like(\Theta^*)}_{\infty,\infty,\infty} = \sup_{ij\ell} |\frac{\partial \like(\Theta^*)}{\partial \Theta_{ij\ell}}|$, by the union bound we have,
\begin{align*}
	\pr \Big( \infnorm{\nabla \like(\Theta^*)} > t \Big)  \le 
	2N^2 p \cdot \exp\Big({-}\frac{n\eps^2t^2}{2L_f^2}\Big)
\end{align*}
Taking $t = (L_f /\eps) \sqrt{c_1 \log(N^2p) /n}$ establishes the result.
\end{proof}

\subsection{Restricted Strong Convexity of $\mc L$}\label{sec:rsc:main}

 Proving  RSC property~\eqref{eq:rsc_first} for a particular choice of $S$ is a major contribution of our work. This is a nontrivial result since it involves uniformly controlling a dependent non-Gaussian empirical process. {Even for i.i.d. samples, the task is challenging since the quantity to be controlled, $\Delta \mapsto \rem \like(\Delta;\Theta^*)$, is a \textit{random function} and one needs a uniform bound on it from below.}
Controlling the behavior of this function becomes significantly harder without the independence assumption. We establish the RSC property for the {negative} log-likelihood loss~\eqref{eq:Likelihood} under the MBP model~\eqref{eq:main_GLM} in the following:

\begin{proposition}\label{prop:rsc:main}
	Let $\sigma_s^2 = \sigma_s^2(\Theta^*)$. There exists a numerical constant $c_1>0$ such that if
	\begin{align}\label{eq:samp:complexity}
		 n   \ge c_1  \frac{ G_f(\Theta^*)}{c_f^2 \,c_\ell^6}\, s^3   \log (N^2p)
	\end{align}
	then, the RSC property~\eqref{eq:rsc_first} holds with
	\begin{align*}
		\kappa = \min\Big\{\frac14 c_f c_\ell^2,1\Big\}, \quad \text{and} \quad \tau^2 =   \frac{\sigma_s^2} s
	\end{align*}
	 for all tensors $\Delta \in \mbb C(S^*;\Theta^*)$, with probability at least $1 - 2 (N^2p)^{ -c_0 s}$. The constant $c_0 = \bigO(c_\ell^{-2})$ in the exponent only depends on $c_\ell$. 
\end{proposition}

\subsubsection{Proof sketch of Proposition~\ref{prop:rsc:main}}	
To prove this result we proceed by a establishing a series of intermediate lemmas.
First, we show that $ \rem \like(\Delta;\Theta^*)$
is lower bounded by a quadratic function of $\Delta$:
\begin{lemma}[Quadratic lower bound]\label{lem:quad:lb}
	The remainder of the first-order Taylor expansion of the negative log-likelihood function, around $\Theta^*$, satisfies
	\begin{align}\label{eq:mcE:def}
		\rem \like(\Delta;\Theta^*) \;\ge\;\; \mcE(\Delta;\mbb X) := \frac{c_f}{n}  \sum_{t=1}^n \sum_{k=1}^N \ip{\Delta_{k\mdot\mdot}, X^{t-1}}^2,
	\end{align}
	for all $\Theta^* \in \Omega$ and $\Delta \in \reals^{N \times N \times p}$. %
\end{lemma}
Notice that  $\Delta \mapsto \mc E(\Delta;\mbb X)$ is a random function due to the randomness in  $\mbb X:=\{x^t\}_{t=-p+1}^n$. %
Next, we show that the population mean of $\Delta \mapsto \mcE(\Delta;\mbb X)$ is strongly convex:
\begin{lemma}[Strong convexity at the population level]\label{lem:RSC:pop}
	Under Assumption~\ref{ass1},
	\begin{align}\label{lem:RSC_cb}
		\Exp\mcE(\Delta;\mbb X)  \geq 
		c_f \, c_\ell^2\, \norm{\Delta}_F^2,\qquad {\rm for\ all\ }\Delta\in\Real^{N\times N\times p}.
	\end{align}	
\end{lemma}
Lemmas~\ref{lem:quad:lb} and~\ref{lem:RSC:pop} are proved in Appendix~\ref{sec:main_results_proof}.
Next, we show that for a fixed $\Delta$, the quantity $\mcE(\Delta;\mbb X)$ concentrates around its mean. Section \ref{sec:concentration} proves and provides more comments on the following concentration inequality:
\begin{proposition}[Concentration inequality]\label{prop:concentration_via_martingale1} 
	For any $\Delta \in \reals^{N \times N \times p}$, if $\mbb X$ is generated as \eqref{eq:main_GLM} with process parameter $\Theta$ satisfying $g_f(\Theta)<1$, then
	\begin{align}\label{eq:main_concentration}
		\Prob \Big( \big |\mc E(\Delta;\mbb X)-\Exp \mc E(\Delta;\mbb X) \big | > t\norm{\Delta}_{2,1,1}^2 \Big) \,\leq\, 2\exp\left({-}\frac{ nt^2}{\Gf}  \right).
	\end{align}
\end{proposition}

Finally, we use the structural properties of  set $\mbb C(S^*;\Theta^*)$ along with Lemma~\ref{lem:RSC:pop} and Proposition~\ref{prop:concentration_via_martingale1} to give a uniform quadratic lower bound on $\mc E(\Delta;\mbb X)$, which holds with high probability:
\begin{lemma}%
	\label{lem:RSC:uniform}
	The result of Proposition~\ref{prop:rsc:main} holds with $R\mc L(\Delta;\mbb X)$ replaced with $\mc E(\Delta;\mbb X)$.
\end{lemma}

The proof of Lemma \ref{lem:RSC:uniform} (see Section \ref{sec:RSC:uniform:proof})  makes use of a discretization argument. Proving uniform laws %
are challenging when the parameter space is not finite. 
The discretization of the set $\mbb C(S^*;\Theta^*)$ uses estimates of the entropy numbers for absolute convex hulls of collections of  points (Lemma~\ref{lem:covering_number}). These estimates are well-known in approximation theory and have been previously adapted to the analysis of regression problems in~\cite{raskutti2011minimax}.
Proposition~\ref{prop:rsc:main} follows by combining from Lemmas~\ref{lem:quad:lb} and \ref{lem:RSC:uniform}. %
\section{Concentration of $\mc E(\Delta;\mbb X)$}
\label{sec:concentration}

{Proposition~\ref{prop:concentration_via_martingale1}} %
  is a concentration inequality for $\Delta \mapsto \mc E(\Delta;\mbb X)$, a quadratic empirical process based on dependent non-Gaussian variables. For independent sub-Gaussian  variables, such a concentration result is often called the {Hanson--Wright} inequality \cite[Thm. 1]{rudelson2013hanson}. Providing such inequalities for dependent random variables is significantly more challenging. For dependent Gaussian variables, the machinery of the Hanson--Wright inequality can still be adapted to derive the desired result; see \cite[Prop. 2.4]{basu2015regularized}. However, these arguments do not extend easily to non-Gaussian dependent variables and hence other techniques are needed to provide such concentration inequalities.

A key observation, discussed in Section~\ref{sec:H_inf_bound}, is that the MBP can be represented as a discrete-space Markov chain. This allows us to use concentration results for dependent processes in countable metric spaces. There are several results for such processes; see \cite{kontorovich2008concentration,marton1996bounding,samson2000concentration} and \cite{kontorovich2012obtaining} for a review. Here, we apply the result by~\cite{kontorovich2008concentration}. These concentration inequalities are stated in terms of various mixing and contraction coefficients of the underlying process.
 The challenge is to control the contraction coefficients in terms of the process parameter $\Theta^*$, which in our case is done using quantities $g_f(\Theta^*)$ and $G_f(\Theta^*)$.

We start by stating the result by Kontorovich et. al.~\cite{kontorovich2008concentration} for a process  $\{X^t\}_{t\in [n]}$ consisting of (possibly dependent) random variables taking values in a countable space $\mc S$. %
 For any $\ell \ge k \ge 1$, define the \emph{mixing coefficient} %
\begin{align}\label{eq:eta:def}
	\eta_{k\ell}\overset{\Delta}
		= \underset{w,w',y}
		\sup\ \Big\| \Prob\left(X_\ell^n=\cdot\mid X_k=w,X_1^{k-1}=y\right)-\Prob\left(X_\ell^n=\cdot\mid X_k=w',X_1^{k-1}=y\right) \Big\|_{\text{TV}}.
\end{align}
where the supremum is over  $w,w'\in\mc S$ and $y\in \mc S^{k-1}$.  Let $\Eta \in \reals^{n \times n}$ be the matrix with entries $\eta_{k\ell}$ for $\ell \ge k$ and zero otherwise (i.e., an upper triangular matrix),  and let $\mnorm{\Eta}_\infty := \max_k \sum_{\ell\geq k}\eta_{ k\ell}$ be the $\ell_\infty$ operator norm of  $\Eta$.   
\begin{proposition}
\label{prop:kontorovich}
\cite[Theorem 1.1]{kontorovich2008concentration}
Let $\phi:\mc S^n\rightarrow \Real$ be an $L$-Lipschitz function of $\mbb X:=\{X^t\}_{t=1}^n$ with respect to the Hamming norm, then $\phi$ concentrates around its mean as,
\begin{align}\label{eq:gen:concent}
	\Prob\big(\, |\phi(\mbb X)-\Exp\phi(\mbb X)|>\eps\, \big)\leq2
	\exp\left({-}\frac{\eps^2}{nL^2 \mnorm{\Eta}_\infty^2}\right), \quad \forall \eps > 0.
\end{align}
\end{proposition}
We apply the above result to $\phi(\mbb X)=\mc E(\Delta;\mbb X)$ by finding an upper bound for the Lipschitz constant $L$ of map $\mbb X\mapsto \mc E(\Delta,\mbb X)$, and bounding the mixing coefficients in terms of the process parameters $\Theta^*$, $p$ and the link function $f$. Before doing so, let us discuss how to bound $ \mnorm{\Eta}_\infty^2$ using more tractable contraction coefficients.

\subsection{Bounding $\norm{\Eta}_\infty$ using Markov contraction}\label{sec:H_inf_bound}

We refer to a time-invariant process over a countable space $\mc S$ as \emph{$p\,$-Markov} if  for some finite $p$,%
\begin{align}
\label{eq:pMarkov}
\Prob\big(x^t = z \mid x^{t-1},x^{t-2},\dots\, \big)
=\Prob\big(x^t  =z  \mid x^{t-1},x^{t-2},\dots, x^{t-p}\,\big), \quad \forall z \in \mc S,\ \ \forall\ t\in\mbb Z.
\end{align}
We start by recalling a well-known contraction quantity, the \emph{Dobrushin ergodicity coefficient}, and relating it to the mixing coefficients  of $p$-Markov processes.

\subsubsection{Dobrushin ergodicity coefficient}
For a  Markov chain (or $1\,$-Markov process) over a discrete space $\mc S$, let  $ P = (P_{ij})\in\Real^{|\mc S|\times|\mc S|}$ be its (transition) kernel. The kernel is a nonnegative stochastic matrix, i.e., each row is a probability distribution. Thus, $P \ge0$ and $P \bfone = \bfone$ where $\bfone \in \reals^{|\mc S|}$ is the all-ones vector. Let
\begin{align}\label{eq:H1:def}
\mc H_1 := \big\{u\in\Real^{|\mc S|}\mid \;\bfone\T u=0\big\}.
\end{align}
This subspace is invariant to every Markov kernels $P\in\Real^{|\mc S|\times|\mc S|}$, \ie, for all $u\in\mc H_1$, we have $u\T P\in\mc H_1$. The \emph{Dobrushin ergodicity coefficient} of $P$ is defined as
\begin{align}\label{eq:Dob:char:1}
\tau_1(P):=\underset{u\in\mc H_1}\sup\ \frac{\onorm{u\T P}}{\onorm{u}}.
\end{align}
It follows from the invariance of $\mc H_1$ to $P$ that
\beq
\label{eq:l_steps}
\onorm{u\T P^\ell}\,\leq\, \tau_1(P)^\ell\onorm{u}\quad\forall u\in\mc H_1.
\eeq
The following alternative characterization is well-known \cite{rhodius1997maximum} (cf.~Appendix \ref{sec:aux:proof} for a proof): %
\begin{lemma}\label{lem:Dob:char:2}
	The Dobrushin ergodicity coefficient of $P$ is given by
	\begin{align}
	\tau_1(P) = \sup_{x,y \, \in \mc S} \tvnorm{(e_x- e_y)^T P }
	\end{align}
	where $e_x$ is the $x$-th basis vector of $\reals^{\mc S}$.
\end{lemma}
Recall that $\norm{\pi_1-\pi_2}_{\rm TV}$ denotes the \emph{total variation} distance between probability distributions $\pi_1$ and $\pi_2$. For discrete distributions we have, 
$\norm{\pi_1-\pi_2}_{\rm TV}=\half\onorm{\pi_1-\pi_2}\leq~1,$
with equality if and only if $\pi_1$ and $\pi_2$ are orthogonal, \ie, have completely mismatched supports.
Consequently, for any stochastic matrix $P$, we have $\tau_1(P)\leq 1$. Furthermore, the inequality is strict if and only if no two rows of $P$ are orthogonal. Markov kernels with $\tau_1(\cdot)<1$ are said to be \emph{scrambling}. A sufficient condition for $\tau_1(P)<1$ is $P$ having at least one column with all entries positive. %

\subsubsection{The $p$-step chain} A $p\,$-Markov process can be equivalently represented by a Markov kernel $\mc K\in [0,1]^{|\mc S|^p\times |\mc S|^p}$
that gives transition probabilities for consecutive blocks of size $p$. For any $t\in\mbb Z$,
\begin{align}\label{eq:pMarkovKernel}
\mc K_{ij} = \Prob\Big((x^t,x^{t-1},\ldots,x^{t-p+1})=j\; \Big| \; (x^{t-1},x^{t-2},\ldots,x^{t-p})=i \Big),\quad\forall\ i,j\in\mc S^p.
\end{align}
Kernel matrix $\mc K$  is constrained since $\mc K_{ij}$ {can be nonzero only if} $(j_2,j_3,\ldots, j_p)=(i_1,i_2,\ldots, i_{p-1})$. The $r$-step chain associated with $\mc K$ has kernel $\mc K^r$. In general, for all $i,j\in\mc S^p$
\begin{align}
{ (\mc K^r)_{ij} } &= \Prob\Big((x^{t+r-1},x^{t+r-2},\ldots,x^{t+r-p})=j\; \Big| \; (x^{t-1},x^{t-2},\ldots,x^{t-p})=i \Big), \; r \ge 1.%
\label{eq:pMarkovKernel:2}
\end{align}
Similarly, $(\mc K^r)_{ij}$ can be nonzero only if $(j_{r+1},j_{r+2},\ldots j_p)=(i_1,i_2,\ldots,i_{p-r}),$ for $r<p$. However, no such constraint applies for $r\geq p$. Moreover, one can verify that for $r<p$, a pair of rows $(\mc K^r)_{i\mdot}$ and $(\mc K^r)_{i'\mdot}$ are always orthogonal for $i,i'\in\mc S^p$ such that $i_1\neq i'_1$. Consequently, $\tau_1(\mc K^r)=1$ for all $r<p$.

Fortunately for $r=p$, one can show that $\tau_1(\mc K^p)<1,$ under the mild assumption that 
\[
	\Prob\big(x_t = z\mid (x^{t-1},x^{t-2},\ldots, x^{t-p}) = j \big)>0\quad \text{for all $z \in \mc S$ and $j \in \mc S^{p}$},
\] since this implies that $\mc K^p$ is a positive matrix and hence \emph{scrambling}. Note that the above condition always holds for the MBP defined in  \eqref{eq:main_GLM}, due to  $f \in [\eps,1-\eps]$; see Assumption~\ref{ass2}.
Specializing~\eqref{eq:pMarkovKernel:2} to $r=p$, for any $t\in\mbb Z$ we have, 
\begin{align}\label{eq:pKernel}
(\mc K^p)_{ij}= \Prob\Big((x^{t+p-1},x^{t+p-2},\ldots,x^{t})=j\; \Big| \; (x^{t-1},x^{t-2},\ldots,x^{t-p})=i \Big),\quad\forall\ i,j\in\mc S^p.
\end{align}

{We make use of $\tau_1(\mc K^p)$, i.e., the ergodicity coefficient of the $p$-step chain, to  bound  $\mnorm{\mbf H}_\infty$ in Proposition~\ref{proof:lem:bounding_ekl}.}  %
Let us define
\begin{align}
F_p(x) := 2 + \frac{2p^2}{(\frac1x - 1)^2}, \quad x \in \reals.
\end{align}

\begin{lemma} \label{lem:bounding_ekl}For a $p$-Markov process over $\mc S$, with equivalent kernel  $\mc K\in\Real^{|\mc S|^p\times|\mc S|^p}$ given by \eqref{eq:pKernel}, the mixing coefficients defined in~\eqref{eq:eta:def} are bounded as
	\begin{align}\label{eq:eta:bound}
	\eta_{k\ell} \;\leq\; \tau_1(\mc K^p)^{1+\lfloor(\ell-k-1)/p)\rfloor}, \quad {\ell\ \geq  k}. 
	\end{align}
	In particular, for any $\tau \in [\tau_1(\mc K^p), 1)$
	\begin{align}\label{eq:Hinf:bound}
	\mnorm{\mbf H}_\infty^2 :=  \Big(\max_{k\in[n]} \sum_{\ell \geq k}\eta_{k\ell} \Big)^2 \;\le\; F_p( \tau ).
	\end{align}
\end{lemma}

Lemma~\ref{lem:bounding_ekl} is a general result that applies to any $p$-Markov chain, including the MBP process considered in this paper. In order to use this result, we need to control $\tau_1(\mc K^p)$. The following result 
provides a bound on the Dobrushin coefficient  of the $p$-step chain derived from the MBP process.

\begin{lemma}\label{lem:bounding_dcc}
	Let the $\mbb X$ be the MBP generated according to \eqref{eq:main_GLM}, where $f:\Real\rightarrow[\eps,1-\eps]$ is $L_f$-Lipschitz for some $\eps\in(0,\half)$. Then, the $p$-step kernel $\mc K^p$ over $\mc S=\bin^N$ satisfies
	\[\tau_1(\mc K^p)
	\; \le \; g_f(\Theta^*).  %
	\]%
\end{lemma}

Armed with these result, we are ready to finish the proof of Proposition~\ref{prop:concentration_via_martingale1}.

\subsection{Proof of Proposition~\ref{prop:concentration_via_martingale1}}

 The following lemma is proven in Appendix \ref{sec:aux:proof}:
\begin{lemma}\label{lem:Lipscthiz_constant}
 The map $\mbb X \mapsto \mcE(\Delta;\mbb X)$ is  Lipschitz with respect to the Hamming distance on $\mc S^{n+p-1}$, with constant %
\begin{align*}
L \le \frac{2c_f}{n}\norm{\Delta}_{2,1,1}^2.
\end{align*}
\end{lemma}

We apply \eqref{eq:Hinf:bound} of~Lemma~\ref{lem:bounding_ekl}, %
with $\tau = g_f(\Theta^*)$. Not that $\tau \in [\tau_1(\mc K^p), 1)$ due to Lemma~\ref{lem:bounding_dcc} and  assumption $g_f(\Theta^*) < 1$. We obtain
\begin{align*}
	\mnorm{\mbf H}_\infty^2  \;\le\; F_p\big( g_f(\Theta^*) \big) = \frac{\Gf}{4c_f^2}
\end{align*}
where the equality is by the definition of $\Gf$. Applying Proposition~\ref{prop:kontorovich} with these estimates of $L$ and $\mnorm{\mbf H}_\infty^2$ and setting $\eps=t\norm{\Delta}_{2,1,1}^2$ in~\eqref{eq:gen:concent} completes the proof.

\section{Proofs of Key Auxiliary Lemmas} 
\label{sec:aux_lemmas_proofs}
Some of the key auxiliary lemmas are proved in this section. The proofs of other lemmas are postponed to the  appendices. %

\subsection{Proof of Lemma~\ref{lem:RSC:uniform}}\label{sec:RSC:uniform:proof}
For this proof we use the notation $\norm{\Delta}_q:=\norm{\Delta}_{q,q,q}$ for the $\ell_q$ norm of a 3-tensor $\Delta\in\Real^{[N]^2\times[p]}$. Note that $\tnorm{\Delta}=\norm{\Delta}_F.$ We also write $\ball_q(r)$ for the tensor $\ell_q$ ball of radius $r$, and $\partial \ball_q(r)$ for the boundary of that ball.  For example,
\begin{align*}
	\ball_1(r) := \{\Delta\in \reals^{N \times N \times p}: \norm{\Delta}_1 \le r\}, \quad \partial \ball_2(r) := \{\Delta\in \reals^{N \times N \times p}: \norm{\Delta}_2 = r\}, \quad 
\end{align*}
 Let $\Gfs = \Gf$ be the quantity defined in \eqref{eq:definition_of_G}. %
Recall that $S^*$ is the support of the best $\ell_1$ approximator of $\Theta^*$ that has cardinality $s$, i.e., the optimal solution to \eqref{eq:compressability}, whereby $\|\Theta^*_{S^{*c}}\|_1=\sigma_s(\Theta^*)$. Let
\begin{align*}
	\Cc^* :=\mbb C(S^*,\Theta^*)=\big \{\Delta\in\Omega^*:\; \norm{\Delta_{S^{*c}}}_{1}\leq3\norm{\Delta_{S^*}}_{1}+4\sigma_s(\Theta^*)\big \}.
\end{align*}
\paragraph{Step~1: Fixed $\ell_2$ norm.} We first establish the RSC (with no tolerance) for tensors in $\Cc^*$ of a given Frobenius norm, say $\tnorm{\Delta}  =r_1$.

Note that for any $\Delta \in \mbb C^*$, we have $\Delta = \Delta_{S^*} + \Delta_{S^{*c}}$, hence
\begin{align*}
	\norm{\Delta}_1 =\;\onorm{\Delta_{S^*}}+\onorm{\Delta_{S^{*c}}}
	\;\le\; 4 \norm{\Delta_{S^*}}_1 + 4 \norm{\Theta_{S^{*c}}}_1 
	\;\le\; 4 \big(\sqrt{s} \norm{\Delta}_F + \sigma_s(\Theta^*) \big)  
\end{align*}
using $\norm{\Delta_{S^*}}_1 \le \sqrt{s} \norm{\Delta_{S^*}}_F $ and $\norm{\Theta_{S^{*c}}}_1 \le \sigma_s(\Theta^*)$. It follows that for any $r_1 > 0$,
\begin{align*}
	\Cc^* \cap \partial \ball_F(r_1) \;\subseteq\; \mbb B_1\big(r_2 \big),
\end{align*}
where $r_2: = 4 \big(r_1 \sqrt{s} + \sigma_s(\Theta^*)\big)$.
 Next we consider covering $\Cc^* \cap \partial \ball_F(r_1)$ by finding an $\eps$-cover of $\ball_1(r_2)$. 

For a metric space $(T,\rho)$, let $\Nc(\eps,T,\rho)$ be the $\eps$-covering number of $T$ in $\rho$. The quantity $\log \Nc(\eps,T,\rho)$ is called the metric entropy. The following is an adaptation of a result of~\cite[Lemma~3, case $q=1$]{raskutti2011minimax}:
\begin{lemma}\label{lem:covering_number}
	Let $\Xd \in \reals^{n \times d}$ be a matrix with column normalization $\norm{\Xd_{*j}}_2 \le \sqrt{n}$ for all $j$. 
	Consider the set of matrices $\reals^{N \times d}$ and let $\ball_1^{N \times d}(u)$ be the (elementwise) $\ell_1$ ball of radius $u$ in that space, i.e.
	\begin{align*}
		\ball_1^{N \times d}(u) := \{D \in \reals^{N \times d} : \norm{D}_1 \le u \}.
	\end{align*}
	Consider the (pseudo) metric $\rho(D_1,D_2):=\frac{1}{\sqrt n}\norm{\msf X(D_1-D_2)\T}_2$ on $\reals^{N \times d}$. Then, for a sufficiently small constant $C_1 > 0$, the metric entropy of $\ball_1(u)$ in $\rho$ is bounded as
	\begin{align*}
		\log \Nc\big(\eps, \ball_1^{N \times d}(u), \rho\big) \leq C_2 \frac{u^2}{\eps^2}\log (Nd),, \quad \forall \eps \le C_1 u.
	\end{align*}
\end{lemma}

Now, recall the design matrix $\Xd \in \reals^{n \times Np}$ defined in~\eqref{eq:Xdesign:def}. Note that $\Xd$ satisfies the column normalization property $\tnorm{ \Xd_{\mdot j}} \leq \sqrt{n}$ for all $j$ since $\Xd_{\mdot j}\in\bin^n$. Fix $\eps \in (0,2C_1 r_2/r_1)$ for sufficiently small $C_1 > 0$. It follows that there exists an $(r_1 \eps/2)$-cover $\Nc''$ of  $\ball_1^{N \times Np}(r_2)$ in the metric defined in Lemma \ref{lem:covering_number} with cardinality bounded as
\begin{align*}
	\log |\Nc''|  \;\lesssim \;\frac{r_2^2}{r_1^2 \eps^2 } \log(N^2 p).
\end{align*}
Define a \textit{stacking operator} $\Sc: \reals^{N \times N \times p} \to \reals^{N \times Np}$ that reshapes a tensor to a matrix as follows:
 \begin{align}\label{eq:stacking:op}
 	\Sc(\Delta) := [\Delta_{\mdot\mdot 1}\ \Delta_{\mdot\mdot 2}\ \ldots \Delta_{\mdot\mdot p}]\in\Real^{N\times Np}.
 \end{align}
Also denote for a set $A$ denote by $\Sc(A)=\{\Sc(a)\mid a\in A\}$. Then we have
\begin{align*}
\Sc(\Cc^* \cap \partial \ball_F(r_1)) \;\subseteq\; \Sc(\mbb B_1\big(r_2 \big)) = \ball_1^{N \times Np}(r_2).
\end{align*}
Define a (pseudo) metric on the tensor space by $\rhob(\Delta,\Delta') = \rho(\Sc(\Delta),\Sc(\Delta'))$. Since $\Sc$ is a bijection, it follows that there is an exterior $(r_1\eps/2)$-covering of $\Cc^* \cap \partial \ball_F(r_1)$ in metric $\rhob$ with the same cardinality as $\Nc''$; call it $\Nc'$. (Here, the exterior covering means that the elements need not belong the set they cover. Elements of $\Nc'$ are tensors in $\ball_1(r_2)$ but not necessarily in $\Cc^* \cap \partial \ball_F(r_1)$.)

 We can pass from $\Nc'$ to an $(r_1 \eps)$-cover, say $\Nc$, of $\Cc^* \cap \partial \ball_F(r_1)$ with cardinality no more than $\Nc'$,  i.e. $|\Nc| \le |\Nc'|$, (see Exercise 4.2.9 in \cite[p.75]{vershynin2018high}). %
 In particular, we have $\Nc \subseteq \Cc^* \cap\mbb B_F(r_1)$.
 
As stated later in Lemma~\ref{lem:stacking}, $\sqrt{\mc E(\Delta;\mbb X)} = c_f^{1/2}\norm{\Xd \,\Sc(\Delta)\T}_F/\sqrt{n}$. Then, by triangle inequality 
 \begin{align*}
 	|\sqrt{\mc E(\Delta;\mbb X)} - \sqrt{\mc E(\Delta';\mbb X)}| \le c_f^{1/2} \, \rhob(\Delta,\Delta'),
 \end{align*}
 for any two tensors $\Delta$ and $\Delta'$.
 Using $(a-b)^2 \ge \frac12 a^2 -b^2$, we have 
 \begin{align*}
 	\mc E(\Delta;\mbb X) \ge \frac12 \mc E(\Delta';\mbb X) - c_f \, \rhob^2(\Delta,\Delta').
 \end{align*}
 If follows that
 \begin{align*}
 	\inf_{\Delta \, \in \, \Cc^* \,\cap\, \partial \ball_F(r_1)} \mc E(\Delta;\mbb X) 
 		\; \ge \; \half \;\inf_{\Delta \, \in \, \Nc } \mc E(\Delta;\mbb X)  - c_f (r_1 \eps)^2
 \end{align*}
 By Proposition~\ref{prop:concentration_via_martingale1} and the union bound, with probability at least $1-2|\Nc|\exp(- nt^2/\Gfs)$, we have
 \begin{align*}
  |\mc E(\Delta;\mbb X)-\ex \mc E(\Delta;\mbb X) \big | \le t\norm{\Delta}_{2,1,1}^2, \quad \forall \Delta \in \Nc.
 \end{align*}
 Since $\Nc \subseteq \Cc^* \cap\mbb B_F(r_1)$, for any $\Delta \in \Nc$ we have $\norm{\Delta}_{2,1,1}^2 \le s \tnorm{\Delta}^2$  %
 and $\tnorm{\Delta} = r_1$. It follows that with the same probability,
 \begin{align*}
 	\mc E(\Delta;\mbb X) \;\ge\; \ex \mc E(\Delta;\mbb X) - t \,s r_1^2  \;\ge\; (c_f c_\ell^2  - t s)\, r_1^2, \quad \forall \Delta \in \Nc
 \end{align*}
 where we have used Lemma~\ref{lem:RSC:pop} in the second inequality. It follows that we the same probability
 \begin{align*}
 	\inf_{\Delta \, \in \, \Cc^* \,\cap\, \partial \ball_F(r_1)} \mc E(\Delta;\mbb X) 
 	\; \ge \;\Big(\half c_f c_\ell^2  - \half t s -c_f \eps^2\Big)\, r_1^2 .
 \end{align*}
  To simplify, let $\sigma_s = \sigma_s(\Theta^*)$.
 Taking $r_1  = (\sigma_s/\sqrt{s}) + 1\{ \sigma_s = 0\} $, we can balance the two terms in $r_2$. We obtain 
 \begin{align*}
 	4 \sqrt{s} \; \le r_2/r_1 \; \le 8 \sqrt{s}.
 \end{align*}
  The constraint on $\eps$ is $\eps \le 2 C_1 (r_2/r_1)$. It is enough to require $\eps \le 8 C_1 \sqrt{s}$. %
  Taking $\eps^2 = \frac18 c_\ell^2$ and assuming that  $s \ge \frac{c_\ell^2}{512 C_1^2}$ satisfies the constraint. Also, taking $t = \frac14 c_f c_\ell^2 / s $, we obtain
 \begin{align*}
 	\pr \Big( \inf_{\Delta \, \in \, \Cc^* \,\cap\, \partial \ball_F(r_1)} \mc E(\Delta;\mbb X) 
 	\; \ge \; \big(\frac14 c_f c_\ell^2\big)  \, r_1^2 \Big)  \ge 1 - 2 \exp \Big(\log |\Nc| - c_f^2 c_\ell^4 \frac{n}{16 s^2 G_f} \Big) =: P_1
 \end{align*}
 Noting that 
 \begin{align*}
 	\log |\Nc| \le C_3 \big(8 \sqrt{s}\big)^2 \Big( \frac8{c_\ell^2} \Big) \log(N^2 p),
 \end{align*}
 the probability is further bounded as 
 \begin{align*}
 	1 - P_1 \;\le\; 2 \exp \Big(  \frac{512}{c_\ell^2} C_3\, s   \log (N^2p) - c_f^2c_\ell^4 \frac{n }{16 s^2 G_f} \Big).
 \end{align*}
 Assuming $ c_f^2c_\ell^4 \frac{n }{16 s^2 G_f} \ge  \; {1024}\, c_\ell^{-2}C_3 \, s   \log (N^2p)$, we have
 \begin{align*}
 1 - P_1 \;\le\; 2 \exp \Big( -512 \, C_3  c_\ell^{-2} \, s   \log (N^2p) \Big) \le 2 (N^2p)^{ -c_1 s}.%
 \end{align*}
 where $c_1 = {O( c_\ell^{-2})}$ only depends on $c_\ell$. Thus, we have established RSC with high probability for tensors in $\Cc^* \,\cap\, \partial \ball_F(r_1)$ with curvature $\kappa = \frac14 c_f c_\ell^2$ and tolerance $\tau^2=0$. 
 
Note that when $\sigma_s = 0$ (i.e., the case of hard sparsity),  $\Cc^*$ is a cone hence the above extends immediately to all $\Delta\in \Cc^*,$ since $\mc E(t\Delta;\mbb X)=t^2\mc E(\Delta;\mbb X)$ for all $t>0,$ thus completing the proof. 
Let us assume $\sigma_s > 0$ in the rest of the proof.

 \paragraph{Step~2: Extending to the complement of the $\ell_2$ norm ball.}  For $\sigma_s > 0$, since $\Cc^*$ is not a cone, we cannot use a scale-invariance argument to extend to general tensors. However, we have the following:
 \begin{lemma}
 	Assume that RSC holds for $\Ec$ in the sense of $\Ec(\Delta;\mbb X) \ge \kappa \fnorm{\Delta}^2$, for all $\Delta \in \Cc^* \cap \partial \ball_F(r)$. Then, RSC holds in the same sense for all $\Delta \in	\Cc^* \cap \{\Delta:\; \fnorm{\Delta} \ge r\}$.
 \end{lemma}
 The lemma establishes the RSC of the previous step for all of $\Cc^* \cap \{\Delta:\; \fnorm{\Delta} \ge r_1\}.$ The proof is straightforward and follows from the observation that $\mc E(\cdot;\mbb X)$ satisfies $\mc E(t\Delta;\mbb X)=t^2\mc E(\Delta;\mbb X),$ for $t\geq 1$.
 
 \paragraph{Step~3: Extending to small radii.} It remains to extend the result to $\Delta \in \Cc^* \cap \{\Delta:\; \fnorm{\Delta} < r_1\}$. In this case, we simply take  $\tau^2 := r_1^2 =   \sigma_s^2 / s$ (since $\sigma_s > 0$ by assumption) so that
 \begin{align*}
 	\Ec(\Delta;\mbb X) \ge 0 \ge \fnorm{\Delta}^2 - \tau^2
 \end{align*}
 so that the RSC holds with curvature $=1$ and tolerance $\tau^2$. Putting the pieces together, we have the RSC for all $\Delta\in \Cc$ with the probability given in Step~1, curvature $\kappa = \min\{ \frac14 c_f c_\ell^2,1\}$ and tolerance $\tau^2 =  \sigma_s^2 / s$. 
This concludes the proof. %

\subsection{Proof of Lemma \ref{lem:bounding_ekl}}
\label{proof:lem:bounding_ekl}
\newcommand\Lc{\mathcal L}
\begin{proof}[Proof of~\eqref{eq:eta:bound}]
Recall that $X^{n}_{-p+2}$ together make $n$ samples of the $p$-Markov process.
Fix $k \ge 1$ and take $z\in\mc S^{k-1}$ and $y\in\mc S^{p-1}$. We use the shorthand $X_{-p+2}^k=wyz,$ to denote $X_k=w, X_{k-p+1}^{k-1}=y$ and $X_{-p+2}^{k-p}=z$ and define the law
\begin{align*}
\Lc^{(\ell \,\to\, n)}_{wyz} &:= \Prob\big(X_\ell^n=\cdot\mid X_{-p+2}^{k}=wyz\big) \\
&=\Prob\big(X_\ell^n=\cdot\mid X_{k-p+1}^k=wy\big) =: 	\Lc^{(\ell \,\to\, n)}_{wy}
\end{align*}
using the $p$-Markov property, showing that $\Lc_{wyz}$ does not depend on $z$. Thus, we also write $	\Lc^{(\ell \,\to\, n)}_{wy}$ for $\Lc^{(\ell \,\to\, n)}_{wyz}$.

\textbf{Case 1.} Assuming $\ell+p \le n$, we have
\begin{align*}
&\Prob\big(X_\ell^n=x_\ell^n \mid X_{k-p+1}^k=wy\big) \\ 
&\qquad=\Prob\big(X_{\ell+p}^n=x_{\ell+p}^n \mid X_\ell^{\ell+p-1}=x_\ell^{\ell+p-1}\big) \cdot
\Prob\big(X_\ell^{\ell+p-1}=x_\ell^{\ell+p-1} \mid X_{k-p+1}^k=wy\big) \\
&\qquad=\phi\big( x_{\ell+p}^n \mid x_\ell^{\ell+p-1}\big) \cdot \psi_{wy}(x_\ell^{\ell+p-1} )
\end{align*}
where we have defined $\phi(u  \mid v ) := \Prob\big(X_{\ell+p}^n= u \mid X_\ell^{\ell+p-1}=v \big)$ and 
\[ \psi_{wy}(\beta) :=  \Prob\big(X_\ell^{\ell+p-1}= \beta \mid X_{k-p+1}^k=wy\big)\]
We note that $\psi_{wy}(\cdot)$ is the $wy$-th row of $\mc K^{\ell+p-k-1}$ which follows by comparing the definition of $\psi_{wy}$ with~\eqref{eq:pMarkovKernel:2} applied with $t=k+1$ and $r=\ell+p-k-1$. Letting $e_i$ denote the $i^{\rm th}$ row of identity in $\Real^{|\mc S|^p\times|\mc S|^p}$, we have
\begin{align*}
\psi_{wy} = e_{wy}^\top \;\mc K^{\ell+p-k-1}.
\end{align*}
Now, we have
\begin{align}\label{eq:tv:reduction}
\begin{split}
2 \tvnorm{\Lc^{(\ell \,\to\, n)}_{wyz} - \Lc^{(\ell \,\to\, n)}_{w'yz}} 	&= \sum_{x_\ell^n}\big | \Prob\big(X_\ell^n=x_\ell^n \mid X_{k-p+1}^k=wy\big)-\Prob\big(X_\ell^n=x_\ell^n \mid X_{k-p+1}^k=w'y\big) \big| \\
&= \sum_{x_\ell^{\ell+p-1}} \; \sum_{x_{\ell+p}^n} 
\phi \big( x_{\ell+p}^n \mid x_\ell^{\ell+p-1}\big) \big|
\psi_{wy}(x_\ell^{\ell+p-1} ) -  \psi_{w'y}(x_\ell^{\ell+p-1} )\big| \\
&= \sum_{x_\ell^{\ell+p-1}} \big|
\psi_{wy}(x_\ell^{\ell+p-1} ) -  \psi_{w'y}(x_\ell^{\ell+p-1} )\big| \\ 
&= \norm{\psi_{wy} - \psi_{w'y}}_1  = 
2 \tvnorm{\Lc^{(\ell \,\to\, \ell+p-1)}_{wy} - \Lc^{(\ell \,\to\, \ell-p+1)}_{w'y}}.
\end{split}
\end{align}
Thus, we have
\begin{align*}
\eta_{k\ell} &= \sup_{w,w',y,z}\tvnorm{\Lc^{(\ell \,\to\, n)}_{wyz} - \Lc^{(\ell \,\to\, n)}_{w'yz}} \\
&=  \half\sup_{w,w',y} \norm{\psi_{wy} - \psi_{w'y}}_1 =\half\sup_{w,w',y} \|(e_{wy}-e_{w'y})\T\mc{K}^{\ell+p-1-k}\|_1.
\end{align*}

Let $m = \ell-k-1$.
Writing $m=p\lfloor m/p\rfloor+(m \ {\rm mod}\ p)$ and using $\half(e_{wy}-e_{w'y})\in\mc H_1$ (see Definition~\eqref{eq:H1:def}), we get
\begin{align}\label{eq:eta_gamma_bound}
\eta_{k\ell} &\leq \underset{v\in \mc H_1}\sup\ \onorm{v\T\mc K^{p+p\lfloor m/p\rfloor+(m \ {\rm mod}\ p)}}
\ {\overset{\rm (a)}\le\underset{v\in \mc H_1}\sup\ \tau_1\left(\mc K^{(m \ {\rm mod}\ p)}\right)\onorm{v\T\mc K^{p+p\lfloor m/p\rfloor}}}
 \nonumber\\
&\overset{\rm (b)}\leq \underset{v\in \mc H_1}\sup\ \onorm{v\T(\mc K^p)^{1+\lfloor m /p\rfloor}} \;\leq\;  \tau_1(\mc K^p)^{1+\lfloor(m /p\rfloor},
\end{align}
where (a) {follows from \eqref{eq:l_steps} applied for $u\T=v\T\mc K^{p+p\lfloor m/p\rfloor}$ which also belongs to $\mc H_1,$} %
while (b) follows from the inequality in Lemma \ref{lem:Dob:char:2}
and the last inequality follows from inequality~\eqref{eq:l_steps} applied for $u=v$. This is the desired result which holds for $\ell + p \le n$.

\textbf{Case 2.} {When $\ell+ p > n$, the reduction in~\eqref{eq:tv:reduction} is unnecessary, i.e., there are fewer than $p$ variables between $\ell$ and $n$. We cannot write the difference of the two underlying laws in terms of rows of $\mc K^r$ for some integer $r$. But, we can augment and consider $\Lc^{(\ell \,\to\, n+u)}_{wyz}$ where $u=\ell+p-n$ and then get  $\Lc^{(\ell \,\to\, n)}$ by marginalization. We have} 
\begin{align*}
\tvnorm{\Lc^{(\ell \,\to\, n)}_{wyz} - \Lc^{(\ell \,\to\, n)}_{w'yz}} \;\le\; 
\tvnorm{\Lc^{(\ell \,\to\, n+u)}_{wyz} - \Lc^{(\ell \,\to\, n+u)}_{w'yz}} 
\end{align*}
since marginalization does not increase the total variation distance. This follows from the triangle inequality: Assuming $p(\cdot,\cdot)$ and $q(\cdot,\cdot)$ to be some probability mass functions,
\begin{align*}
\sum_x \Big| \sum_y p(x,y) - \sum_y q(x,y)\Big| \le \sum_x \sum_y |p(x,y) - q(x,y)|.
\end{align*}
Since $\ell+p = n+u$, the proof in this case reduces to that of Case 1. The proof of~\eqref{eq:eta:bound} is complete.

\end{proof}

\begin{proof}[Proof of~\eqref{eq:Hinf:bound}]
It is enough to prove the inequality for $\tau = \tau_1(\mc K^p)$. Then, the result follows since $F_p$ is increasing on $[\tau_1(\mc K^p),1)$. For this $\tau$, we have for any fixed $k$ (recalling $\eta_{kk} = 1$),
\begin{align*}
\sum_{\ell\geq k}\eta_{k\ell}
\;\leq\;  {1 + \sum_{\ell > k} \tau^{1+\lfloor(\ell-k-1)/p\rfloor} 
	\;\le} \; 1+ {\sum_{m\geq 1}\;\sum_{\ell=(m-1)p+k+1}^{mp + k}\tau^m }
= 1+\frac{p\tau}{1-\tau}.
\end{align*}
It follows that
\begin{align*}
\mnorm{\mbf H}_\infty^2 := \Big( \max_k \sum_{\ell\geq k}\eta_{k\ell}\Big)^2 \le 
\Big( 1+\frac{p\tau}{1-\tau}\Big)^2 \le 2 + 2\frac{p^2\tau^2}{(1-\tau)^2}
\end{align*}
which is the desired result.

\end{proof}

\subsection{Proof of Lemma \ref{lem:bounding_dcc} }
\label{proof:lem:bounding_dcc}
By Lemma~\ref{lem:Dob:char:2}, $\tau_1(\mc K^p)=\sup_{z,y\,\in\,\mc S^p}  
\tvnorm{\Prob_z-\Prob_y}$ and by Pinsker's inequality \cite[p. 44]{csiszar2011information},
\[
\tau_1^2(\mc K^p)=\sup_{z,y\in\mc S^p}\tvnorm{\Prob_{z}-\Prob_{y}}^2\leq\ \sup_{z,y\in\mc S^p}\half D_{\rm KL}(\Prob_{z}\|\Prob_{y}).
\]
We derive the bound in Lemma \ref{lem:bounding_dcc} using an upper bound for $D_{KL}(\Prob_z,\Prob_y)$.
We start by defining some notation useful for this proof. Let
\begin{align*}
\Prob_z(\cdot) \;:=\; \Prob\big(X_t^{t+p-1}=\cdot\mid X_{t-p}^{t-1}=z \big) 
\;=\; \Prob\big(X_1^{p}=\cdot\mid X_{1-p}^{0}=z\big) 
\end{align*}
which holds for any $t$ since the process is  time-invariant. We also write $p_z(\cdot)$ for the density of~$\pr_z$ with respect to the counting measure on $\mc S^{p}$ (i.e., the probability mass function of $\pr_z$),~\ie, \begin{align*}p_z(a) := \Prob(X_t^{t+p-1}=a \mid X_{t-p}^{t-1}=z)= \pr(X_1^p = a \mid X_{1-p}^0 = z),\qquad \forall\ a \in \mc S^p.\end{align*} %
We also let,
\begin{align*}
q(\xi \mid b) := \pr\big(x_t = \xi \mid X^{t-1}_{t-p} = b\big),\qquad\xi \in \mc S,\ \ b\in\mc S^p,
\end{align*}
which is also invariant to $t$, and define
\begin{align*}
d_K(b,b') := \dkl\big( q(\cdot \mid b) \;\|\; q(\cdot \mid b')\big),\qquad b,b'\in\mc S^p,
\end{align*}
where $\dkl$ denotes the KL divergence. The following Lemma gives a decomposition for the KL divergence between two samples of a $p$-Markov process. Lemmas \ref{lem:KL:decomp} and \ref{lem:bern:kl:upbound} are proved later in Appendix \ref{sec:aux:proof}.

\begin{lemma}\label{lem:KL:decomp}
	Assume that the process is $p$-Markov in the sense of~\eqref{eq:pMarkov}. Then,
	\begin{align*}
	D_{\rm KL}(\Prob_{z}\,\|\,\Prob_{y})  =
	\sum_{t=1}^p \ex_z\,  \Big[ d_K\big( (X_1^{t-1}, z_{t-p}^0) \; \| \; (X_1^{t-1}, y_{t-p}^0)\big) \Big].
	\end{align*}
\end{lemma}

\noindent We also note the following bound on the KL divergence between Bernoulli random variables. %

\begin{lemma}\label{lem:bern:kl:upbound}
	Let $U\sim {\rm Ber}(p),$ and $V\sim{\rm Ber}(q)$ for $p,q \in [\eps,1-\eps]$ for some $\eps \in (0,\half)$. Then,
	\begin{align*}
	D_{\rm KL}(U\| V) =	p \log \frac{p}{q} + (1-p) \log\frac{1-p}{1-q} \; \le \; \frac{3}{4\eps(1-\eps)} (p-q)^2.
	\end{align*}	
\end{lemma}

\noindent Continuing with the proof of Lemma \ref{lem:bounding_dcc}, recall from~\eqref{eq:main_GLM} that  $\mc S = \{0,1\}^N$, and
\begin{align*}
x^t \mid X_{t-p}^{t-1} \;\sim\; \prod_{i=1}^N {\rm Ber}\left(f\left(\inner{\Theta_{i \mdot \mdot},X^{t-1}_{t-p}}\right) \right).
\end{align*}
Let $\alpha_i^t = \inner{\Theta_{i \mdot \mdot}, (X_1^{t-1}, z_{t-p}^0)}$ and $\beta_i^t  = \inner{\Theta_{i \mdot \mdot}, (X_1^{t-1}, y_{t-p}^0)}$. Then,
\begin{align*}
d_K\big( (X_1^{t-1}, z_{t-p}^0) \; \| \; (X_1^{t-1}, y_{t-p}^0)\big)
&= \sum_{i=1}^N ,
\dkl \left( {\rm Ber} \big(f(\alpha_i^t) \big) \; \big\| \; {\rm Ber} \big(f(\beta_i^t) \big)  \,\right), \\
&\le \frac{3}{4\eps(1-\eps)} \sum_{i=1}^N \big[f(\alpha_i^t) - f(\beta_i^t) \big]^2\quad  \le \frac{3 L_f^2}{4\eps(1-\eps)} \sum_{i=1}^N \big(\alpha_i^t - \beta_i^t \big)^2,
\end{align*}
where the first inequality is by Lemma~\ref{lem:bern:kl:upbound} and the second by Lipschitz assumption on $f$. Now,
\begin{align*}
\alpha_i^t - \beta_i^t = \inner{\Theta_{i \mdot \mdot}, (0, z_{t-p}^0 - y_{t-p}^0)} 
= \sum_{\ell=t}^p \inner{\Theta_{i \mdot \ell}, z_{t-\ell} - y_{t-\ell}}.
\end{align*}
We have $ |\alpha_i^t - \beta_i^t| \le \sum_{\ell = t}^p \norm{\Theta_{i \mdot \ell}}_1 \cdot \norm{z_{t-\ell} - y_{t-\ell}}_\infty$. Since the latter $\ell_\infty$ norm is at most $1$, we conclude from Lemma~\ref{lem:KL:decomp} that
\begin{align*}
D_{\rm KL}(\Prob_{z}\,\|\,\Prob_{y})  \le\sum_{i=1}^N \frac{3L_f^2}{4\eps(1-\eps)}\sum_{t=1}^p (\alpha^t_i-\beta^t_i)^2 \le
\frac{3 L_f^2}{4\eps(1-\eps)}  \sum_{t=1}^p   \sum_{i=1}^N \Big(  \sum_{\ell = t}^p \norm{\Theta_{i \mdot \ell}}_1  \Big)^2.
\end{align*} 
Since by assumption $\eps < 1/2$, the result follows.

\section{Discussion}%
\label{sec:conclusion}
We  analyzed a sparse multivariate Bernoulli process that evolves as an autoregressive GLM. This model provides a framework for identifying the spiking characteristics and the underlying network of an ensemble of neurons from a finite sample of their spike trains. The conditional probability of spiking for each neuron is modeled as a GLM with a link function, based on the history of the multivariate process. Our work extends the model considered by \cite{kazemipour2017robust} to the multi-dimensional setting and provides the first result in the high-dimensional regime for this model. 

We provided the first rigorous analysis of the regularized maximum likelihood estimator for the 
MBP model \eqref{eq:main_GLM} with (i) an approximately sparse parameter, (ii) a Lipschitz, log-concave non-linear inverse link function, and (iii) for the general case of $p\geq 1$ and $N\geq 1$. We proved that the estimator achieves the fast rate $\bigO(s\log(N^2p)/n)$ under exact sparsity, and the slow rate $\bigO(\wt \tau^2_s(\Theta^*)\sqrt{\log(N^2p)/n})$ under approximate sparsity, where $\wt\tau_s(\Theta^*)$ is a measure of the $\ell_1$ approximation error.

Our error bounds are valid under fairly general assumptions on the stability and wide-sense stationarity of the process. We also highlighted key properties of the nonlinear inverse link function that can provide control on the estimation error and the sample complexity. In addition, we identified a new norm $G_f(\Theta^*)$ %
 of the true parameter $\Theta^*$, defined in \eqref{eq:definition_of_g}, that measures the hardness of the estimation problem.
For example, we demonstrated that when the parameters are assumed to decay fast with the lag, \ie, $|\Theta^*_{ij\ell}|=\bigO(\ell^{-\alpha})$ or $\bigO(e^{-\beta\ell})$, a $\bigO(p^3)$  improvement can be achieved {in terms of the sample complexity}.

Our result in general requires a sample complexity of $n = \Omega(s^3\log (N^2p))$, allowing a high-dimensional scaling in $(N,p)$. The dependence on the sparsity  parameter ``$s$'' may be suboptimal, as indicated by our simulations. However, Figure~\ref{fig:sparsity_sample_error_grid} hints that the optimal dependence on $s$, may  be superlinear for this model.

The main obstacle in obtaining a better guarantee, on the sample complexity, is in showing the restricted strong convexity of the negative log-likelihood function. This is a challenging task as the proof of Proposition~\ref{prop:rsc:main} suggests. To be precise, if one is able to show a concentration inequality of the form in Proposition \ref{prop:concentration_via_martingale1} albeit for deviations of the form $t\norm{\Delta}^2_F$ instead of $t\norm{\Delta}_{2,1,1}^2$, then one would achieve a sample complexity of $\Omega(s\log(N^2p))$. However, it remains open whether such concentration is true.

Our analysis for establishing the RSC is novel 
and the proof techniques can be applied to any multivariate processes over discrete countable spaces with long-range dependence. Mixing properties of higher-order Markov chains are also discussed during the proof which may be of independent interest. %
The concentration inequality in Proposition \ref{prop:concentration_via_martingale1} paves the way for obtaining error bounds for other maximum likelihood estimators with decomposable regularizers for the multilag multivariate Bernoulli model. Of particular interest are the \textit{network-of-fiber} models that use the group norm regularizer $\norm{\Theta}_{1,1,2}$.
 
\paragraph{Acknowledgements.}

P. Pandit, M. Sahraee and A.K. Fletcher were supported in part by the National Science Foundation under Grants 1738285 and 1738286 and the Office of Naval Research under Grant N00014-15-1-2677. S. Rangan was supported in part by the National Science Foundation under Grants 1116589, 1302336, and 1547332, and the industrial affiliates of NYU WIRELESS. 
\bibliographystyle{alpha}
\newcommand{\etalchar}[1]{$^{#1}$}

\begin{appendices}

\section{Intermediate proofs from Section \ref{sec:sketch}}
\label{sec:main_results_proof}

{To simplify the notation, for a 3-tensor $\Delta$, we often write $\norm{\Delta}_{q} = \norm{\Delta}_{q,q,q}$ for all $q \in (0, \infty]$.} With this notation $\tnorm{\Delta}=\norm{\Delta}_F$ and we use them interchangably. Recall $S^*$ is the set that optimizes \eqref{eq:compressability} and we refer $\sigma_s(\Theta^*)$ as $\sigma_s$ when unambiguous from context. The following alternative expressions will be used frequently:
\begin{align}\label{eq:E_forms}
\ip{\Delta_{k\mdot\mdot}, X^{t-1}} = 
\sum_{\alpha=1}^p \ip{\Delta_{k\mdot\alpha}, X_{\mdot \alpha}^{t-1}}  =: 
\sum_{\alpha=1}^p \ip{\Delta_{k\mdot\alpha}, x^{t-\alpha}}.
\end{align}

\subsection{Proofs of Lemmas \ref{lem:quad:lb} and \ref{lem:RSC:pop}}\label{sec:establishing_RSC}

Here and elsewhere, we use $\mbb X$ as a shorthand for the collection of $X^{t-1},\; t=1,\dots,n$ or equivalently for
\begin{align}\label{eq:X:def}
\mbb X := (x^t)_{-p+1}^{n-1} = (x^{n-1},x^{n-2},\dots,x^{-p+1}) \in \mc S^{n+p-1},
\end{align}
where $ \mc S:=\{0,1\}^N$.

Let us record an alternative form for $\mcE(\Delta;\mbb X)$, which will be useful in subsequent analysis. 
Let $\msf X\in\bin^{n\times Np}$ be the design matrix with $t^{\rm th}$ row,
\begin{align}\label{eq:Xdesign:def}
	\msf X_{t\mdot}:=[(x^{t-1})^\top \ (x^{t-2})^\top \ldots (x^{t-p})^\top] \in \bin^{1 \times Np}
\end{align}
 and recall from \eqref{eq:stacking:op} the \textit{stacking operator} $\Sc: \reals^{N \times N \times p} \to \reals^{N \times Np}$ that reshapes a tensor $\Delta\in\Real^{N\times N\times p}$ as:
 \begin{align*}
 	\Sc(\Delta) := [\Delta_{\mdot\mdot 1}\ \Delta_{\mdot\mdot 2}\ \ldots \Delta_{\mdot\mdot p}]\in\Real^{N\times Np}.
 \end{align*}
 Recall the definition of the quadratic form $\mc E(\Delta;\mbb X)$ from~\eqref{eq:mcE:def}. We state the following alternative representation without proof:
 \begin{lemma}\label{lem:stacking}
 	For any $\Delta \in \reals^{N \times N \times p}$, we have 
 	\begin{align}
 		\mc E(\Delta;\mbb X) = \frac{c_f}{n}\norm{\msf X \,\Sc(\Delta)\T}^2_F
 			=\frac{c_f}n \sum_{t=1}^n \norm{\Xd_{t*} \Sc(\Delta)\T}^2_2 %
 	\end{align}
 \end{lemma}

\begin{subsubsection}{Proof of Lemma~\ref{lem:quad:lb}}
Recall that the loss can be written as 
\[
\like(\Theta) = - \frac1n \sum_{i=1}^N \sum_{t=1}^n \ell_{it}( \ip{\Theta_{i\mdot\mdot}, X^{t-1}}) %
\]
where $\ell_{it}(u) = x_i^t \log f(u) + (1-x_i^t) \log(1-f(u))$ is the likelihood for $x_{i}^t\sim{\rm Ber}(f(u))$. We have
\begin{align*}
\frac{\partial^2 \like }{\partial \Theta_{abc}\partial\Theta_{k \ell m}} (\Theta)=- \frac1n  \sum_t \ell_{a t}''( \ip{\Theta_{a\mdot\mdot}, X^{t-1}}) \, X_{bc}^{t-1} X_{\ell m}^{t-1} \,1\{k = a\},
\end{align*}
{where $-\ell_{at}''(u) \ge c_f > 0$ for all $u \in [\eps,1-\eps]$ by Assumption~\ref{ass2}.}
It follows that
\begin{align*}
\frac{\partial^2 \like}{\partial \Theta_{abc}\partial\Theta_{k \ell m}}(\Theta) \ge
\frac{c_f}{n}  \sum_{t=1}^n\, X_{bc}^{t-1} X_{\ell m}^{t-1} \,1\{k = a\}, \quad \forall\, \Theta \in \Omega.
\end{align*}
The Hessian of the quadratic form is uniformly controlled from below as,
\begin{align*}
\ip{ \Delta \nabla^2\like(\Theta), \Delta}&\ge
\sum_{a,b,c,k,\ell,m}\Delta_{abc} \Big[\frac{c_f}{n}  \sum_{t=1}^n\, X_{bc}^{t-1} X_{\ell m}^{t-1} \,1\{k = a\} \Big] \Delta_{k\ell m},   \\
&=\sum_{k,b,c,\ell,m}\Delta_{kbc} \Big[\frac{c_f}{n}  \sum_{t=1}^n\, X_{bc}^{t-1} X_{\ell m}^{t-1}  \Big] \Delta_{k\ell m} \\
&= \frac{c_f}{n} \sum_{k=1}^N \sum_{t=1}^n \Big( \sum_{b,c} \Delta_{kbc} X_{bc}^{t-1}\Big) 
\Big( \sum_{\ell,m} \Delta_{k\ell m} X_{\ell m}^{t-1}\Big)  
=\ \frac{c_f}{n}  \sum_{t=1}^n \sum_{k=1}^N \ip{\Delta_{k\mdot\mdot}, X^{t-1}}^2, %
\end{align*}
for all $\Theta \in \Omega$ and $\Delta \in \reals^{N \times N \times p}$.  The proof is complete.
\end{subsubsection}

\subsubsection{Proof of Lemma~\ref{lem:RSC:pop}}
Using the notation in~\eqref{eq:Xdesign:def} and~\eqref{eq:stacking:op}, %
Lemma~\ref{lem:stacking} implies
\[
	\Exp\mcE(\Delta;\mbb X)= %
	c_f  \ex \norm{\Xd_{t*} \Sc(\Delta)\T}^2_2 \quad \text{for all $t$},
\]
since by assumption  the process is stationary (i.e., the  distribution of $\Xd_{t*}$ is the same for all~$t$).
Let  $D:=\Sc(\Delta)\in\Real^{N\times Np}$ for simplicity, and let $\bfR := \Exp\msf X_{t\mdot}\T  \msf X_{t\mdot}\in\Real^{Np\times Np}$ be the population autocorrelation matrix, again independent of $t$ by stationarity. Then,
\begin{align*}
		\ex \mcE(\Delta;\mbb X)= c_f \,\ex \norm{\Xd_{t*}D\T}^2_2 
		= c_f \ex \tr\big( \Xd_{t*}^\top \Xd_{t*} D^\top D\big) = c_f \tr\big(\bfR D^\top D \big).
\end{align*}
Since $\bfR - \lambda_{\min}(\bfR)I\succeq 0,$ and $D \T D \succeq 0$, we have that 
\begin{equation}
\label{eq:mean_RSC}
	\Exp\mc E(\Delta;\mbb X) \;\geq\; c_f\lambda_{\min}(\bfR) \tr(D^\top D)= c_f\lambda_{\min}(\bfR)\,\norm{D}_F^2.
\end{equation}
We note that $\bfR$ %
is a block symmetric matrix with blocks 
{$\bfR_{ij}:=\ex [x^{t-i} (x^{t-j})^\top] \in \Real^{N\times N}$.} 
We also note that due to the stationarity, $\bfR_{ij}$ only depends on $i-j$, so with some abuse of notation we write $\bfR_{ij} = \bfR_{i-j}$, i.e., $\bfR$ is block Toeplitz. Let ${\bf C}_{i-j}$ denote the centered autocorrelation matrix $\Exp[(x^{t-i}-\Exp x^t)(x^{t-j}-\Exp x^t)\T]$, whereby ${\bf R}_{i-j}={\bf C}_{i-j}+\Exp x^t(\Exp x^t)\T.$ Define ${\bf C}$ similarly as a block Toeplitz matrix with ${\bf C}_{ij}={\bf C}_{i-j}$.  Consequently $\lambda_{\rm min}(\bf R)\geq \lambda_{\rm min}({\bf C}).$

Let $\mc X(\omega) \in \mathbb{C}^{N\times N}$ be the power spectrum matrix of the process as in definition~\eqref{eq:power:spec:def} so that 
\begin{align}
{ {\bf C}_{\ell}}:= {\frac{1}{2\pi}}{\int_{-\pi}^\pi}  \mc X(\omega) \,e^{j\omega\ell}{d\omega},
\end{align}
Recall from~\eqref{eq:lamin:psd} that %
\begin{align}
c_\ell^2 := \min_{\omega\,\in\,[-\pi,\pi)}\,  \lambda_{\min}(\mc X(\omega)) > 0.
\end{align}
It is well-known that $\lambda_{\min}({\bf C}) \geq c_\ell^2$. See for example~\cite[Proposition 2.3]{basu2015regularized} or \cite[Lemma~4.1]{gray2006toeplitz}. For completeness, we prove this assertion below. This together with equation \eqref{eq:mean_RSC} and $\fnorm{D}^2 = \fnorm{\Delta}^2$ proves Lemma~\ref{lem:RSC:pop}.

\medskip
\emph{Proof of $\lambda_{\min}({\bf C}) \geq c_\ell^2$.} Fix $\bfu\T = \begin{bmatrix}
u_0\T \  u_2\T \  \ldots \  u_{p-1}\T
\end{bmatrix},$ where $u_i \in \Real^N$ and set $G(\omega) =\frac1{\sqrt{2\pi}} \sum_{r=0}^{p-1} u_r e^{-jr\omega}.$ Then,  
\begin{align}
	\bfu\T{\bf C}\bfu 
	= \sum_{r,s=0}^{p-1} u_r\T {\bf C}_{r-s} u_s 
	&= \sum_{r,s=0}^{p-1}  u_r\T  \Big[ \frac1{2\pi}\int_{-\pi}^\pi \mc X(\omega)e^{j(r-s)\omega} d\omega  \Big] u_s \notag \\
	&=   \int_{-\pi}^\pi G\herm(\omega)\mc X(\omega)G(\omega)d\omega. \label{eq:uRu}
\end{align}
Since $\mc X(\omega)$ is a Hermitian matrix, $G\herm(\omega)\mc X(\omega)G(\omega)$ is always a real matrix. Moreover, we have that 
\[
	G\herm(\omega)\mc X(\omega)G(\omega) 
	\geq \lambda_{\min}(\mc X(\omega))\,G\herm(\omega)G(\omega) \ge c_\ell^2 \,G\herm(\omega)G(\omega)
\]
 hence
\[
	\bfu\T{\bf C}\bfu 
		\geq c_\ell^2 \int_{-\pi}^\pi G\herm(\omega)G(\omega)d\omega 
		= c_\ell^2 \sum_{r,s=0}^{p-1} u_r\T ( \delta_{r-s}I_N) u_s = c_\ell^2  \tnorm{\bfu}^2,
\] by Parseval's theorem. (Alternatively, reverse the operation in~\eqref{eq:uRu} with $\mc X(\omega) = 1 \cdot I_N$ and recall that the inverse of a flat spectrum is the delta function). Here, $\delta_x = 1\{x = 0\}$. Taking the minimum over $\norm{\bfu}_2 =1$ completes the proof.

\section{Intermediate proofs from Sections \ref{sec:concentration} and \ref{sec:aux_lemmas_proofs}}
\label{sec:aux:proof}

\begin{proof}[Proof of Lemma \ref{lem:Dob:char:2}] Optimization problem in \eqref{eq:Dob:char:1} is scale invariant, hence, %
\begin{align}
\tau_1(P) = \underset{u \,\in\, \mc H_1(1)}\sup\ \onorm{u\T P}, \label{eq:tau_def_char1}
\end{align}
where $\mc H_1(1)=\{u\in\mc H_1\mid \onorm{u}\leq 1\}.$
We will show that the set $\mc H_1(1)=C:=\conv(\{\half(e_x-e_y)\})$. Using this, \eqref{eq:tau_def_char1} is a maximization of a convex function $\onorm{u\T P}$ over a polytope with extreme points $\half(e_x-e_y), x,y \in \mc S$. It follows that the maximum occurs, at least, at an extreme point, which gives the desired result. The inequality in the statement of the lemma follows since the total-variation is bounded by 1.

The rest of the proof establishes $\mc H_1(1)=C$. %
The inclusion $C\subseteq \mc H_1(1)$ can be verified easily by checking the membership of extreme points of $C$ in $\mc H_1(1)$, since $\mc H_1(1)$ is a convex set. 
We now prove the nontrivial direction $\mc H_1(1)\subseteq C$. 

Let the ambient space be $\Real^m$, $\Delta_m$  the probability simplex in $\Real^m$, and $\partial \ball_1 := \{u\in \reals^m: \norm{u}_1 = 1\}$ the boundary of $\ell_1$ ball. We have $C = \half\Delta_m+\half(-\Delta_m),$ which is a Minkowski sum. This follows since taking the Minkowski sum and taking the convex hull commute~\cite[Theorem~3]{krein1940regularly}. Hence, it suffices to show that for any vector $u\in\mc H_1(1)$, there exists a pair of probability vectors $\pi_1,\pi_2\in\Delta_m$ such that $u=\half(\pi_1-\pi_2)$. 
Since $0 \in C$, and $\mc H_1(1) = \conv(0, \partial \ball_1\cap\,\mc H_1)$, 
it is enough to consider  $u \in \partial \ball_1\cap\,\mc H_1$. %

Let $u \in \partial \ball_1\cap\,\mc H_1$, and let $u_+$ and $u_-$ be the positive and negative parts of $u$, that is, $(u_+)_i = \max(u_i,0)$ and $(u_-)_i = -\min(u_i,0)$. Taking $\pi_1 = 2 u_+$ and $\pi_2 = 2 u_-$, we have $u = \frac12 (\pi_1 - \pi_2)$. Also, due to $u\in\partial\ball_1$, $1 = \norm{u}_1 = \frac12 \norm{\pi_1}_1 + \frac12 \norm{\pi_2}_1$ whereas due to $u\in\mc H_1$, $0= \bfone\T u = \frac12 \norm{\pi}_1 - \norm{\pi_2}_1$. It follows that $\norm{\pi_1}_1=\norm{\pi_2}_1 = 1$, that is, $\pi_1,\pi_2 \in \Delta_m$. This concludes the proof.

\end{proof}

\begin{proof}[Proof of Lemma \ref{lem:Lipscthiz_constant}]
	It is enough to consider two sequences $\{x^t\}$ and $\{y^t\}$ which differ in one coordinate, say $\mbb X=(x^{-p+1},x^{-p},\dots,x^{n-1})$ and $\mbb Y = (x^{-p+1},x^{-p},\dots,y^{r},\dots,x^{n-1})$, where $r$ will be fixed.  The general case follows, via triangle inequality, since any $\mbb Y$ can be reached from $\mbb X$ by a sequence $\mbb X =: \mbb X_{(0)}, \mbb X_{(1)},\dots,\mbb X_{(h)} := \mbb Y$ where $h$ is the hamming distance of $\mbb X$ and $\mbb Y$ in $\mc S^{n+p-1}$, such that $\mbb X_{(i)}$ and $\mbb X_{(i-1)}$ are Hamming distance 1 apart, for $i=1,2,\ldots h$. 

	Let $X^{t-1}$ and $Y^{t-1}$ be defined based on $\mbb X$ and $\mbb Y$ as before, i.e., the corresponding $p$-lag history at time $t-1$. Note that $X^{t-1}$ and $Y^{t-1}$ are different only for $t$ such that $t \in \{r+1,\dots,r+p\}$, and for such $r$, we have recalling  the expressions in~\eqref{eq:E_forms},
	\begin{align*}
		|\ip{\Delta_{k\mdot\mdot}, X^{t-1} - Y^{t-1}}| 
		\le \norm{\Delta_{k,*,t-r}}_1 .
	\end{align*}
	We also have
	\begin{align*}
		|\ip{\Delta_{k\mdot\mdot}, X^{t-1} + Y^{t-1}} | \le 2 \norm{\Delta_{k**}}_1.
	\end{align*}
	Combining we obtain
	\begin{align*}
		| \mcE(\Delta;\mbb X) - \mcE(\Delta;\mbb Y) | &= 
		\frac{c_f}{n}  \Big| \sum_{t=r+1}^{r+p} \sum_{k=1}^N \big[ \ip{\Delta_{k\mdot\mdot}, X^{t-1}}^2  -  \ip{\Delta_{k\mdot\mdot}, Y^{t-1}}^2 \big] \Big| \\
	&\le \frac{2c_f}{n}   \sum_{t=r+1}^{r+p} \sum_{k=1}^N   \norm{\Delta_{k,*,t-r}}_1 \, \norm{\Delta_{k**}}_1 \\
		&\;\le \; \frac{2c_f}{n}    \sum_{k=1}^N   \, \norm{\Delta_{k**}}^2_1  =\frac{2c_f}{n}\norm{\Delta}_{2,1,1}^2
	\end{align*}
	where we have used $\sum_{t=r+1}^{r+p} \norm{\Delta_{k,*,t-r}}_1= \norm{\Delta_{k**}}_1$. This proves the claim. %
\end{proof}

\begin{proof}[Proof of Lemma \ref{lem:KL:decomp}]
		Recall the notation $X_1^{p} = (x_p,\dots,x_1)$.  Similarly, let $a = (a_p,\dots,a_1) \in \mc S^p$ so that $X_1^p = a$ is the same as $X_u = a_u$ for all $u \in [p]$.  We also write $a_1^{t-1} = (a_{t-1},\dots,a_1)$ and so on for elements of $\mc S^p$. For any $a, z \in \mc S^p$, we have
		\begin{align*}
		p_z(a) &= \pr(X_1^p = a \mid X_{1-p}^0 = z) \\
		&= \prod_{t=1}^p \pr(x_t = a_t \mid X_{1}^{t-1} = a_1^{t-1},\, X_{t-p}^0 = z_{t-p}^0) \\
		&= \prod_{t=1}^p \pr\big(x_t = a_t \; \big| \; X_{t-p}^{t-1} = (a_1^{t-1}, z_{t-p}^0) \big) = \prod_{t=1}^p q(a_t \mid (a_1^{t-1}, z_{t-p}^0) )
		\end{align*}
		where the second line is by the Markov property. Replacing $a$ with a random variable $X_1^p \in \mc S^p$, 
		\begin{align*}
		p_z(X_1^p) =  \prod_{t=1}^p q\big(x_t \mid (X_1^{t-1}, z_{t-p}^0) \big).
		\end{align*}

		Letting $\ex_z$ denote the expectation assuming $X_1^p\,\sim\, \Prob_{z}$, we have
		\begin{align*}
		D_{\rm KL}(\Prob_{z}\,\|\,\Prob_{y}) 
		&= \ex_z \,\log\frac{p_{z}(X_1^p)}{p_{y}(X_1^p)} \\
		&= \sum_{t=1}^p \ex_z \,\log \frac{q\big(x_t \mid (X_1^{t-1}, z_{t-p}^0) \big)}{q\big(x_t \mid (X_1^{t-1}, y_{t-p}^0) \big)} \\
		&= \sum_{t=1}^p \ex_z  \ex_z\Big[\log \frac{q\big(x_t \mid (X_1^{t-1}, z_{t-p}^0) \big)}{q\big(x_t \mid (X_1^{t-1}, y_{t-p}^0) \big)} \; \Big|\;  X_1^{t-1} \Big] \\
		&= \sum_{t=1}^p \ex_z\,  d_K\big( (X_1^{t-1}, z_{t-p}^0) \; \| \; (X_1^{t-1}, y_{t-p}^0)\big) 
		\end{align*}
		where the last line follows by noting that under $X_1^p\,\sim\, \Prob_{z}$, further conditioning on $X_{1}^{t-1}$ is equivalent to conditioning on $X_{1}^{t-1}$ and $X_{t-p}^0 = z_{t-p}^0$, i.e., $x_t$ will have distribution ${q(\,\cdot\mid  (X_1^{t-1}, z_{t-p}^0) )}$ under this conditioning.
\end{proof}

\begin{proof}[Proof of Lemma~\ref{lem:bern:kl:upbound}]
	It is enough to prove for the case $q \ge p$ (the other case follows by applying the proven case to $1-p$ and $1-q$). The second claim follows from the decomposition of the KL divergence for product distributions. Let $\delta := \eps(1-\eps)$. Fix $p$ and consider the function
	\begin{align*}
	f(q) = p \log \frac{p}{q} + (1-p) \log\frac{1-p}{1-q}  - \frac{1}{4\delta} (p-q)^2,
	\end{align*}
	over $q \in [p, 1-\eps]$. We have 
	\begin{align*}
	f'(q) = (q-p) \Big( \frac{1}{q(1-q)} - \frac{1}{2\delta} \Big).
	\end{align*}
	We have $f(q) = f(p) + f'(\qt) (q-p)$ for some $\qt \in [p,q]$. Note that $f(p) = 0$ and
	\begin{align*}
	f'(\qt) \le (\qt - p) \Big( \frac{1}{\delta} - \frac{1}{2\delta} \Big) \le \frac{1}{2\delta} (q-p)
	\end{align*}
	using the fact that $(\qt(1-\qt))^{-1} \in [4,\delta^{-1}]$. Thus, we have $f(q) \le (q-p)^2/(2\delta)$.%
\end{proof}

\section{Scaling of $g_f(\Theta)$ with $p$}\label{sec:scaling_of_G}
This appendix provides the details for Table~\ref{tab:pscale}. The case of exponential decay follows from that of polynomial decay for which we have:
\begin{lemma}\label{lem:bound_decay_param}
If $|\Theta_{ijk}| \leq C_{ij}\cdot k^{-\alpha}$, for some $\alpha > \frac{3}{2}$. Then %
\[
\bigg(\sum_{\ell=1}^p\sum_{i=1}^N
\Big(\sum_{j=1}^N\sum_{k=\ell}^p|\Theta_{ijk}|\Big)^2\bigg)^{1/2} \le \; c(\alpha) \,\|C\|_{2,1}
\]
for some constant $c(\alpha) > 0$ only dependent on $\alpha$.
\end{lemma}
\begin{proof}
	By approximating the summation with an integral, for any $\alpha  > 1$ (and $\ell \ge 1$), we have
	\[\sum_{k=\ell}^\infty k^{-\alpha} \leq c_1(\alpha) \,\ell^{1-\alpha}.\]
	It follows that $	\sum_{j=1}^N\sum_{k=\ell}^p|\Theta_{ijk}| \le c_1(\alpha) \norm{C_{i*}}_1 \ell^{1-\alpha}$. Hence,
	\begin{align*}
	\sum_{\ell=1}^p\sum_{i=1}^N
	\Big(\sum_{j=1}^N\sum_{k=\ell}^p|\Theta_{ijk}|\Big)^2 &\le c_1^2(\alpha) 
		\sum_{\ell=1}^p\sum_{i=1}^N \norm{C_{i*}}^2_1 \ell^{2-2\alpha} 
		\\&= c_1^2(\alpha) \norm{C}^2_{2,1}\sum_{\ell=1}^p  \ell^{2-2\alpha} \\
		&\le c_1^2(\alpha) c_1(2-2\alpha)\,  \norm{C}^2_{2,1} =: c^2(\alpha)\,\norm{C}^2_{2,1}
	\end{align*}
	assuming that $2-2\alpha < -1$.
\end{proof}

\end{appendices}
 
\end{document}